\documentclass[12pt]{article}
\usepackage[top=1in, bottom=1in, left=1in, right=1in]{geometry}

\usepackage{xcolor}
\usepackage{geometry}
\usepackage{amssymb,bm}
\usepackage{amsmath}
\usepackage{amsthm}
\usepackage{graphicx}
\usepackage{hyperref}
\usepackage{caption}
\usepackage{float} 
\hypersetup{pdfborder=0 0 0}
\usepackage{bbm}
\usepackage{subcaption}

\newtheorem{theorem}{{\sc Theorem}}[section]
\newtheorem{proposition}[theorem]{{\sc Proposition}}
\newtheorem{lemma}[theorem]{{\sc Lemma}}
\newtheorem{corollary}[theorem]{Corollary}
\newtheorem{remark}[theorem]{Remark}
\newtheorem{definition}[theorem]{Definition}
\newtheorem{conjecture}[theorem]{Conjecture}

\newcommand\restr[2]{{ \left.\kern-\nulldelimiterspace #1 \vphantom{\big|} \right|_{#2}}}
\newcommand{\RR}{\mathbb{R}}
\newcommand{\ZZ}{\mathbb{Z}}
\newcommand{\CC}{\mathbb{C}}

\newcommand{\CD}{\mathcal{D}}

\newcommand{\veps}{\varepsilon}

\newcommand{\CL}{\mathcal{L}}

\newcommand{\n}{\noindent}
\newcommand{\comment}[1]{}
\renewcommand{\div}{\mathrm{div}}

\DeclareMathAlphabet{\pazocal}{OMS}{zplm}{m}{n}
\newcommand{\CT}{\pazocal{T}}

\renewcommand{\CD}{\pazocal{D}}

\newcommand{\CB}{\pazocal{B}}
\newcommand{\CG}{\pazocal{G}}

\title{On the lack of external response of a nonlinear medium in the second-harmonic generation process}
\author{Fioralba Cakoni\footnote{Department of Mathematics, Rutgers University, New Brunswick, NJ, USA (fc292@math.rutgers.edu)},  \ \ Narek Hovsepyan\footnote{Department of Mathematics, Rutgers University, New Brunswick, NJ, USA (nh507@rutgers.edu)}, \\ \\ Matti Lassas\footnote{Department of Mathematics and Statistics, University of Helsinki, Finland (matti.lassas@helsinki.fi)} \ \ and \ \ Michael S. Vogelius\footnote{Department of Mathematics, Rutgers University, New Brunswick, NJ, USA (vogelius@math.rutgers.edu)}}
\date{}

\begin{document}
\maketitle

\begin{abstract}
This paper concerns the scattering problem for a  nonlinear medium of compact support, $D$, with second-harmonic generation. Such a medium, when probed with monochromatic light beams at frequency $\omega$, generates additional waves  at  frequency $2\omega$. The response of the medium is governed by a system of two coupled semilinear partial differential equations for the electric fields  at frequency  $\omega$ and $2\omega$.  We investigate whether there are situations in which the generated $2\omega$ wave is localized inside $D$, that is,  the nonlinear interaction of the medium with the probing wave is invisible to an outside observer. This leads to the analysis of a semilinear elliptic system formulated in $D$ with non-standard boundary conditions. The analysis presented here sets up a mathematical framework needed to investigate a multitude of questions related to nonlinear scattering with second-harmonic generation. 
\end{abstract}

\section{Introduction}
\setcounter{equation}{0}
The focus of this paper is a nonlinear model describing higher order optical harmonics  generation in bulk crystals \cite{boyd, mol}. Such media, when probed with monochromatic laser beams, generate waves at new frequencies. The most common example is perhaps the green laser pointer, which emits a frequency-doubled green light based on an infrared laser source and an internal (second-harmonic generating) crystal. In our model the light propagation through the medium is governed by a system of monochromatic Maxwell's equations at different frequencies, which are coupled through sources depending nonlinearly on the time-harmonic electric fields at different frequencies.  To make this more precise: Maxwell's equations,  in terms of the electric field $E$ for a non-magnetic medium (i.e., $B=H$) with no free charge and no current, assert (cf. \cite{boyd})
\begin{equation*}
\nabla (\div E) - \Delta E = -\mu_0 \veps_0 E_{tt} - \mu_0 P_{tt}.
\end{equation*}

\n In linear optics of isotropic source-free media, the polarization relation $P=\veps_0 \chi E$  implies that $\div E = 0$. In nonlinear optics, this linear relation between $P$ and $E$ does not hold, nevertheless, it is still commonly  assumed that $\div E \approx 0$. According to \cite{boyd} this is true in many cases of interest, e.g., in the context of the slowly varying amplitude approximation. This justifies considering the nonlinear vector wave equation 

\begin{equation} \label{wave}
\Delta E - \frac{1}{c^2}  E_{tt} = \mu_0 P_{tt}, 
\end{equation}
where  $c=(\mu_0 \veps_0)^{-1/2}$ is the speed of light, and $\veps_0$ and $\mu_0$ are the vacuum electric permittivity and the magnetic permeability, respectively. We split the polarization into its linear and nonlinear contributions $P = P^L + P^{NL}$  in terms of dependence on $E$.  In our model, we assume that $E,P$ can be decomposed into sums of discrete frequency components (we use $\hat{\cdot}$ to denote a quantity in the frequency domain)
$$E(t,x) = \sum_{n\in \ZZ} \hat{E}(x;\omega_n) e^{-i\omega_n t} \qquad \qquad  P^{L (NL)}(t,x) = \sum_{n\in \ZZ}  \hat{P}^{L (NL)} (x;\omega_n)e^{-i\omega_n t}$$
For the linear polarization the usual constitutive relation of linear media 
$$\hat{P}^{L} (x; \omega_n) = \veps_0 \chi^{(1)}(x;\omega_n) \hat{E}(x;\omega_n)$$ holds, where the linear susceptibility $\chi^{(1)}$ is in general  a matrix. Substituting these expressions into \eqref{wave} we arrive at

\begin{equation} \label{E_n PDE}
\Delta \hat{E}(x;\omega_n)+ \frac{\omega_n^2}{c^2} \veps(x,\omega_n)\hat{E}(x;\omega_n) = -\mu_0 \omega_n^2 \hat{P} (x;\omega_n)
\end{equation}
where $\veps =  \text{Id} + \chi^{(1)}$ is the relative permittivity matrix for the inhomogeneous media and $\hat{P} (x;\omega_n):=\hat{P}^{NL}(x;\omega_n)$ now refers to the nonlinear polarization.  
\begin{remark}
\normalfont
The physical optical field, $E$, is real, and we consider waves with positive frequencies $\omega_n> 0$, $n=1,2,...$. Complexifying we can write $E = \sum_{n> 0} E_n e^{-i\omega_n t} + E_n^* e^{i\omega_n t}$, where $\cdot^*$ denotes the complex conjugate. If we introduce $\omega_{-n}=-\omega_n$, the relevant expansions for $E$ therefore take the form
\begin{equation*}
E = \sum_{n\in \ZZ}  \hat{E}(x;\omega_n) e^{-i\omega_n t},
\end{equation*}
where $\hat{E}(x;0)=0$, the frequencies appear together with their negatives, and the amplitudes obey the symmetry 
\begin{equation} \label{sym}
\hat{E}(\cdot\,;-\omega_n) = \hat{E}^* (\cdot \, ;\omega_n)~.
\end{equation}
\end{remark}

\bigskip

\n In this work we limit ourselves to  second order nonlinear effects, which means that the nonlinear polarization $\hat{P}^{NL}$  depends on $\hat{E}$ quadratically, and more specifically we consider  the sum frequency generation process. In other words , we model the interaction of three waves with frequencies linked through the equation $\omega_3 = \omega_1 + \omega_2$ using a second order nonlinearity. Initially we assume that the two input frequencies are different, $\omega_1 \neq \omega_2$, and that all frequencies are nonzero. The interaction is quantified by the second order susceptibility tensor $\chi^{(2)}$. We define the components of the second order susceptibility tensor $\chi^{(2)}_{ijk}$ as the proportionality constants in the formula

\begin{equation} \label{chi^2 sum permutation}
\hat{P}_i(\omega_3) = \veps_0 \sum_{jk} \sum_{(12)} \chi^{(2)}_{ijk}(\omega_3; \omega_1, \omega_2) \hat{E}_j(\omega_1) \hat{E}_k(\omega_2), 
\end{equation}

\n where $i=1,2,3$ indicates the component of the vector $\hat{P}(\omega_3)$, and similarly for the indices $j,k$ and the fields $\hat{E}(\omega_1),~ \hat{E}(\omega_2)$. Note that here for simplicity we have suppressed the $x$-dependence in the fields. The first argument of $\chi^{(2)}$ is redundant, since it is always the sum of the second and third arguments, but by convention it is still listed \cite{boyd}. Since $\hat{E}_j(\omega_1)$ is associated with the time dependence $e^{-i\omega_1 t}$ and $\hat{E}_k(\omega_2)$ with $e^{-i\omega_2 t}$, their product is associated with $e^{-i\omega_3 t}$, which explains the notation. The second summation over $(12)$, indicates that $\omega_3$ is held fixed and the sum is over the two distinct permutations of $\omega_1, \omega_2$, leading to the formula

\begin{equation}\label{in}
\hat{P}_i(\omega_3) = \veps_0 \sum_{jk} \left[ \chi^{(2)}_{ijk}(\omega_3; \omega_1, \omega_2) + \chi^{(2)}_{ikj}(\omega_3; \omega_2, \omega_1)\right] \hat{E}_j(\omega_1) \hat{E}_k(\omega_2). 
\end{equation}
\n Using the convention known as intrinsic permutation symmetry, we assume that
\begin{equation} \label{intrinsic sym}
\chi^{(2)}_{ijk}(\omega_3; \omega_1, \omega_2) = \chi^{(2)}_{ikj}(\omega_3; \omega_2, \omega_1),
\end{equation}
\n which is motivated by convenience and the fact that only the sum of the two coefficients can be measured in practice, so we may as well assume that the individual pieces are the same. Finally, the symmetry \eqref{sym} for $\hat{E}$ and hence also for $\hat{P}$ implies

\begin{equation} \label{real sym}
\chi^{(2)}_{ijk}(-\omega_3; -\omega_1, -\omega_2) = (\chi^{(2)}_{ijk})^*(\omega_3; \omega_1, \omega_2).
\end{equation}
\n With this in mind, we can now write (\ref{in}) in a more compact tensor notation
\begin{equation*}
\hat{P}(\omega_3) = 2\veps_0 \chi^{(2)}(\omega_3; \omega_1, \omega_2) \hat{E}(\omega_1) \hat{E}(\omega_2).  
\end{equation*}

\n To have a complete description of interaction of the three involved waves we also need the polarizations $\hat{P}(\omega_1)$ and $\hat{P}(\omega_2)$. These are given by analogous formulas, for example since $\omega_1 = \omega_3 - \omega_2$, hence
$$\hat{P}(\omega_1) = 2\veps_0 \chi^{(2)}(\omega_1; \omega_3, -\omega_2) \hat{E}(\omega_3) \hat{E}^*(\omega_2),$$
and similarly
$$\hat{P}(\omega_2) = 2\veps_0 \chi^{(2)}(\omega_2; \omega_3, -\omega_1) \hat{E}(\omega_3) \hat{E}^*(\omega_1).$$

\n Using the above constitutive laws for polarizations in the PDEs \eqref{E_n PDE}, we obtain the governing non-linear system that describes the interaction between the three waves  $\hat{E}(\cdot\,;\omega_1)$, $\hat{E}(\cdot\,;\omega_2)$ and $\hat{E}(\cdot\,;\omega_3)$ in  the  sum frequency generation model.
\begin{remark} \mbox{}
\normalfont
\n In the governing system the three  tensors $\chi^{(2)}$, which can be different and $x$-dependent, must be specified. Each tensor contains $27$ complex constants. So to describe the nonlinear interaction of three waves, 81 constants must be specified. In practice, however, this number is much smaller due to various crystal symmetries (see e.g.  \cite{boyd, mol}).  We also remark that  nonlinear susceptibility tensors obey the Kramers-Kronig relations describing their dependence on generating frequencies (see e.g. \cite[Section 1.7]{boyd}).

\end{remark}

\bigskip

\n In this work we will further limit ourselves to the (degenerate) case when $\omega_1 = \omega_2 =: \omega$ and hence $\omega_3 = 2\omega$. Note that for $\hat{P}(\omega_3)$ the second sum in \eqref{chi^2 sum permutation} drops, as there is in this case only one permutation of $\omega_1, \omega_2$. However, the sum stays for $\hat{P}(\omega_1)$, altogether yielding
\begin{equation} \label{P SHG}
\begin{split}
\hat{P}(\omega) &= 2\veps_0 \chi^{(2)}(\omega; 2\omega, -\omega) \hat{E}(2\omega) \hat{E}^*(\omega)
\\
\hat{P}(2\omega) &= \veps_0 \chi^{(2)}(2\omega; \omega, \omega) \hat{E}(\omega) \hat{E}(\omega).
\end{split}
\end{equation}

\n Let $E_1(x) = \hat{E}(x; \omega)$ and $E_2(x)=\hat{E}(x; 2\omega)$ (to simplify the notation we dropped $\hat{\cdot}$), and denote
$$\chi_1(x,\omega):=\frac{2}{c^2} \chi^{(2)}(\omega; 2\omega, -\omega) \qquad \qquad \chi_2(x,\omega):=\frac{4}{c^2} \chi^{(2)}(2\omega; \omega, \omega).$$
Taking into account that $\chi_1$ (which describes the interaction of a $2\omega$ wave and a $-\omega$ wave) and $\chi_2$ (which describes the interaction of two $\omega$ waves) in general may be different and $x$-dependent, we arrive at the {\it second-harmonic generation}  governing system of PDEs
\begin{equation} \label{SHG}
\begin{cases}
\displaystyle \Delta E_1 + \frac{\omega^2}{c^2} \veps(x, \omega) E_1  = -\omega^2\chi_1(x,\omega) E_2 E_1^*,
\\[.15in]
\displaystyle \Delta E_2 + \frac{4 \omega^2}{c^2} \veps(x, 2\omega) E_2  = -\omega^2\chi_2(x,\omega) E_1 E_1.
\end{cases}
\end{equation}

\n
{\bf {The Scattering Problem for Second-Harmonic Generation.}} We are interested in the interaction of a monochromatic time harmonic incident electrical field with a nonlinear optical medium of bounded support modeled by the second-harmonic generation as described above. We assume that the nonlinear medium occupies a bounded connected region $D\subset {\mathbb R}^3$ with smooth boundary $\partial D$ where $\nu$ denotes the outward normal vector.  We further assume that this medium is probed by a time harmonic incident field at frequency $\omega$ given by $\tilde{E}^{\text{in}}(x,t)=E^{\text{in}}_1(x) e^{-i\omega t}$, where the spacially dependent part $E^{\text{in}}_1$ is an entire solution of 
\begin{equation}\label{inc}
\Delta E^{\text{in}}_1 + \frac{\omega^2}{c^2} E^{\text{in}}_1 = 0, \qquad \qquad \text{in} \ \RR^3.
\end{equation}
\begin{figure}[h]
\center
\captionsetup{width=.90\linewidth}
\includegraphics[scale=0.5]{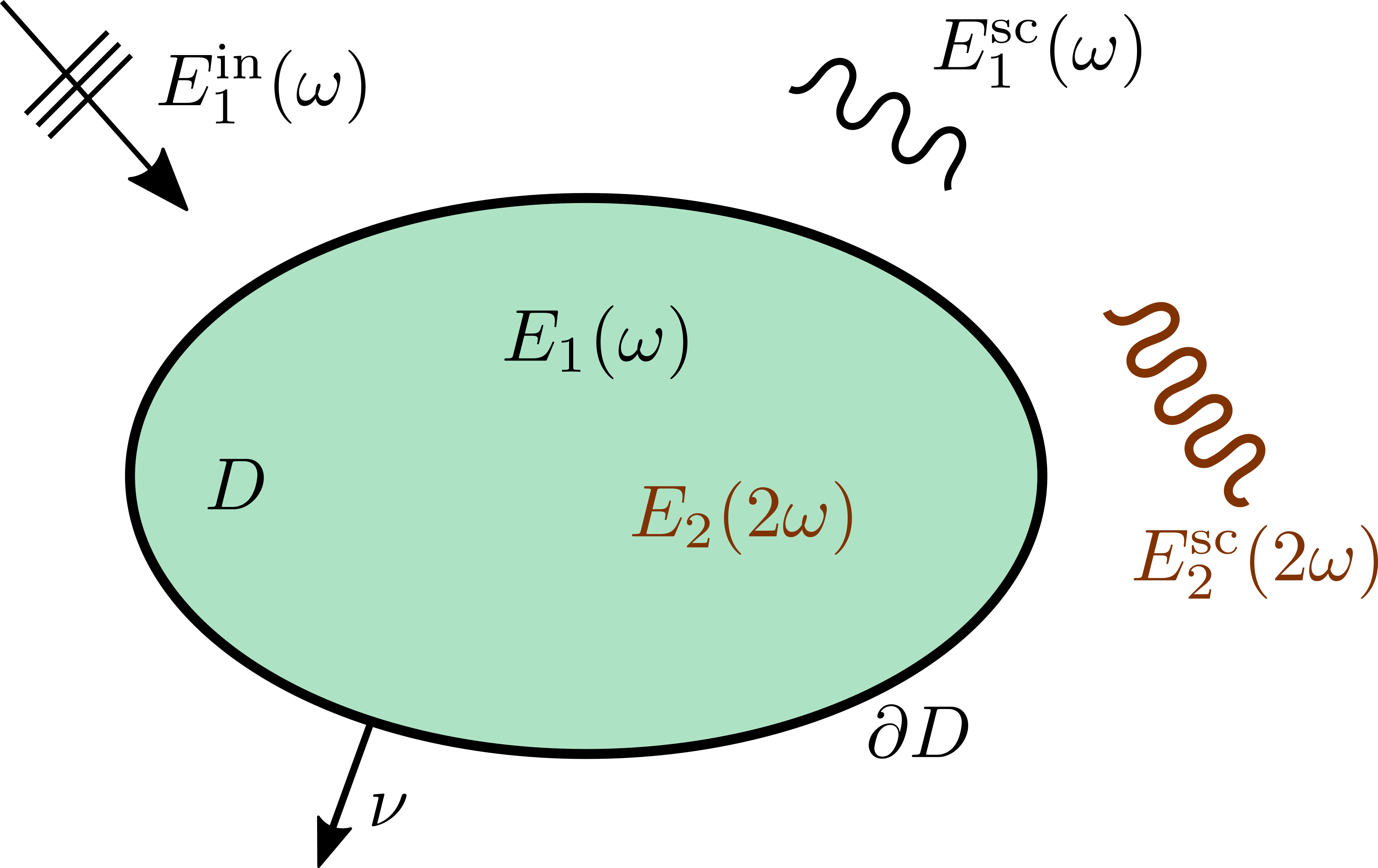}
\caption{The scattering of a monochromatic time-harmonic incident wave by a second-harmonic generation nonlinear optical medium of bounded support $D$.}
\label{FIG cloak}
\end{figure}
\n
The incident field generates a transmitted field $\tilde{E}^{\text{tr}}(x,t)$ in $D$ and a scattered field  $\tilde{E}^{\text{sc}}(x,t)$ in $\RR^3 \backslash \overline{D}$, which due to the nonlinear response of the medium are sums of two time-harmonic waves at frequencies $\omega$ and $2\omega$, that is 
\begin{equation*}
\tilde{E}^{\text{tr}}(x,t) = E_1(x) e^{-i\omega t} + E_2(x) e^{-i2\omega t}, \qquad \qquad
\tilde{E}^{\text{sc}}(x,t) = E_1^{\text{sc}}(x) e^{-i\omega t} + E_2^{\text{sc}}(x) e^{-i2\omega t}.
\end{equation*}
The components  $E_1$ and $E_2$ of the total field  inside $D$ satisfy
\begin{equation} \label{trans}
\begin{cases}
\Delta E_1 + \displaystyle{\frac{\omega^2}{c^2} \veps(x,\omega) E_1}  = -\omega^2\chi_1(x,\omega) E_2 E_1^*, \hspace{0.7in} &\text{in} \ D,
\\[.2in]
\displaystyle{\Delta E_2 + \frac{4 \omega^2}{c^2} \veps(x,2\omega) E_2}  = -\omega^2\chi_2(x,\omega) E_1 E_1, &\text{in} \ D,
\end{cases}
\end{equation}
whereas the components $E_1^{\text{sc}}$ and $E_2^{\text{sc}}$ of the scattered field satisfy
\begin{equation} \label{scattered}
\begin{cases}
\displaystyle{\Delta E_1^{\text{sc}} + \frac{\omega^2}{c^2}  E_1^{\text{sc}}} = 0, &\text{in} \ \RR^3 \backslash \overline{D},
\\[.1in]
\displaystyle{\Delta E_2^{\text{sc}} + \frac{4\omega^2}{c^2}  E_2^{\text{sc}}} = 0, &\text{in} \ \RR^3 \backslash \overline{D}.
\end{cases}
\end{equation}
Furthermore, both the trace and normal derivatives of  the total fields inside and outside $D$ corresponding to the same frequency are continuous across $\partial D$, {\it i.e.}, the following transmission conditions hold
\begin{equation} \label{trans-cond}
\left\{
\begin{array}{rrcllll}
E_1 = E^{\text{in}}_1 + E_1^{\text{sc}} &\qquad \mbox{and} \qquad& \displaystyle{\frac{\partial E_1}{\partial \nu} = \frac{\partial E^{\text{in}}_1}{\partial \nu} + \frac{\partial E_1^{\text{sc}}}{\partial \nu}}  &\qquad  \text{on} \ \partial D,
\\[.2in]
E_2 = E_2^{\text{sc}}\;\;\; \;\;\;  &\qquad \mbox{and} \qquad&  \displaystyle{\frac{\partial E_2}{\partial \nu} = \frac{\partial  E_2^{\text{sc}}}{\partial \nu} } \;\;&\qquad \text{on} \ \partial D.
\end{array}
\right.
\end{equation}
Finally the scattered fields $E_1^{\text{sc}}$, $E_2^{\text{sc}}$ are radiating, {\it i.e.}, they satisfy the Sommerfeld radiation condition, which in three dimension requires
\begin{equation}\label{Somer}
\displaystyle{\lim_{r\to\infty}r\left( \frac{\partial E_1^{\text{sc}}}{\partial r}-i\frac{\omega}{c} E_1^{\text{sc}}\right)=0\qquad \mbox{and}\qquad \displaystyle\lim_{r\to\infty}r\left( \frac{\partial E_2^{\text{sc}}}{\partial r}-i\frac{2\omega}{c} E_2^{\text{sc}}\right)=0},
\end{equation}
uniformly with respect to $\hat{x}:=x/|x|$. In nonlinear optics, the equations (\ref{inc})-(\ref{Somer}) define the scattering problem for a  nonlinear medium $D$ of second-harmonic generation type, with a monochromatic time-harmonic incident field. 

Although higher order harmonic generation in nonlinear optics is investigated in the physics  and material science literature (see e.g. \cite{boyd, mol}), a rigorous mathematical analysis of many questions related to such models is still in its infancy \cite{AZ21, RS23}. There is, however, a vast mathematical literature on direct and inverse problems for semilinear scalar wave equations in the time-domain, or at single frequency (e.g. with Kerr nonlinearities \cite{kerr3}); we refer the reader to \cite{non2,  plamen, DeFilippis_2023, non3, gries, non1, KurLasUhl} and references therein for samples of this work.  

In this paper we investigate, among other things, the existence of probing frequencies  $\omega$ that may yield a vanishing $2\omega$ scattered field, given a second-harmonic generation inhomogeneity of compact support. In other words, at such frequencies the nonlinear effects of the medium can be invisible to an external observer. These frequencies are conceptually somewhat similar to non-scattering wave numbers and/or  transmission eigenvalues in linear scattering theory, {\it i.e.},  probing frequencies at which it is possible to have zero scattering from a given linear inhomogeneity \cite{blas2, blas, CBMS2, nonscattering, nonscattering-A, cak-xiao, Hu, nonscattering-SS}. Generally speaking, the analysis presented here introduces a  mathematical framework needed to investigate a multitude of questions related  to the  nonlinear scattering problem associated with a second-harmonic generation process.

\section{Formulation of the Problem and our Main Results}
\setcounter{equation}{0}

\n
As described above, if we require that  $E_2^{\text{sc}}\equiv 0$ in $\RR^3 \backslash \overline{D}$ in the scattering problem (\ref{inc})-(\ref{Somer}), then we obtain  the second-harmonic non-scattering problem. We introduce the notation that
\begin{eqnarray*}
&&\hskip -10pt \hbox{\bf (N1)~} \hbox{A frequency } \omega, \hbox{ for which } (\ref{inc})-(\ref{Somer}) \hbox{ has a non-trivial solution with } E_2^{\text{sc}}\equiv 0  \\
&&\hskip 30pt \hbox{ in } \RR^3 \backslash \overline{D} \hbox{ is called a second-harmonic non-scattering frequency}.
\end{eqnarray*}
In the same way as we arrive at the transmission eigenvalue problem for the study of non-scattering wave numbers in linear scattering theory (see e.g. \cite{CBMS2}), here we consider the equation for the incident field (\ref{inc}) only in $D$ and denote $V:=E^{\text{in}}_1|_D$. Noting that $E_1-V$ represents the $\omega$-scattered field inside $D$ and letting $q(x, \omega):=\veps(x,\omega)/{c^2}$, we obtain the following nonlinear eigenvalue problem 
\begin{equation} \label{transmission}
\hspace*{-0.2in} \begin{cases}
\Delta E_1 + \omega^2 q(x,\omega) E_1  = -\omega^2\chi_1(x,\omega) E_2 E_1^* \; &\text{in} \ D
\\[.1in]
\Delta E_2 + 4 \omega^2 q(x, 2 \omega) E_2  = -\omega^2\chi_2(x,\omega) E_1^2 &\text{in} \ D
\\[.1in]
\Delta V + \frac{\omega^2}{c^2}  V = 0 &\text{in} \ D
\\[.1in]
 \displaystyle{\frac{\partial (E_1-V)}{\partial \nu}}  = {\mathcal N}_\omega(E_1-V) &\text{on} \ \partial D
\\[.1in]
E_2 =  \displaystyle{\frac{\partial E_2}{\partial \nu}}  = 0 &\text{on} \ \partial D~.
\end{cases}
\end{equation}
Here ${\mathcal N}_\omega$ is the exterior Dirichlet-to-Neumann operator 
$${\mathcal N}_\omega: F \mapsto \restr{\displaystyle{\frac{\partial W}{\partial \nu}}}{\partial D}$$
with $W$ being the unique outwardly radiating solution of 
$$\Delta W+ \frac{\omega^2}{c^2}  W = 0 \;\; \text{in} \ \RR^3 \backslash \overline{D}, \qquad  W=F \;\;  \text{on} \ \partial D.$$
\begin{eqnarray*}
&&\hskip -10pt \hbox{\bf (N2)~} \hbox{We call an  } \omega \hbox{ for which } (\ref{transmission}) \hbox{ has a nontrivial solution a second-harmonic }\\
&&\hskip 100pt \hbox{transmission eigenvalue}. 
\end{eqnarray*}
Being a second-harmonic transmission eigenvalue is a necessary condition for being a second-harmonic non-scattering frequency. Furthermore if $V$ can be extended as an entire solution to the Helmholtz then it is also a sufficient condition. The extendability is likely related to the regularity of $D$ and the constitutive material properties of the inhomogeneity, as is the case for linear scattering (see \cite{nonscattering, nonscattering-A,  nonscattering-SS} . It would be very interesting to study the nonlinear second-harmonic transmission eigenvalue problem, and possibly also its inhomogeneous version where the Cauchy data $E_2, {\partial E_2}/{\partial \nu}$ may be prescribed in order to control the field at frequency $2\omega$ outside $D$.  

\n
However, we are not in the position to analyze (\ref{transmission})  at this time, and instead here we investigate  the solvability of a ``weaker" problem, which provides necessary conditions for $\omega$ being a second-harmonic transmission eigenvalue (and therefore also a necessary condition for $\omega$ being a second-harmonic non-scattering frequency). More specifically, we study the existence of nontrivial solutions $E_1$ and $E_2$ to
\begin{equation} \label{transmission 2}
\begin{cases}
\displaystyle \Delta E_1 + \omega^2 q(x,\omega)  E_1  = -\omega^2\chi_1(x,\omega) E_2 E_1^* \hspace{0.7in} &\text{in} \ D
\\[.15in]
\displaystyle \Delta E_2 + 4 \omega^2 q(x, 2 \omega) E_2   = -\omega^2\chi_2(x,\omega) E_1^2  &\text{in} \ D\\[.15in]
E_2 =  \displaystyle{\frac{\partial E_2}{\partial \nu}}  = 0 &\text{on} \ \partial D.
\end{cases}
\end{equation}

\n  As we discarded the incident field in the above system and restricted the $\omega$ and $2\omega$ fields inside $D$, \eqref{transmission 2} also applies to the situation where there is a $\omega$-source outside the region $D$. Specifically, \eqref{transmission 2} also provides necessary conditions for the vanishing of the second-harmonic wave outside $D$, given an exterior $\omega$-source.

\begin{eqnarray*}
&&\hskip -10pt \hbox{\bf (N3)~} \hbox{We call an } \omega \hbox{ for which } (\ref{transmission 2}) \hbox{ has a nontrivial solution a generalized second-harmonic} \\
&&\hskip 50pt \hbox{transmission eigenvalue, or for short: a generalized transmission eigenvalue}.
\end{eqnarray*}
Note that if at a generalized transmission eigenvalue, the field $E_1$ can be extended to all of $\RR^3$ as sum of an outwardly radiating and an incident entire solution of the Helmholtz equation, then this eigenvalue is a second-harmonic non-scattering frequency. 

In this paper we study  the scalar version of (\ref{transmission 2}),  {\it i.e.}, we assume  $E_1$ and $E_2$ are scalar functions, which we in the following will denote by $u_1$ and $u_2$, respectively. This assumption simplifies  many technicalities in the presentation, but will allow us to capture the mathematical structure associated with these nonlinearities for our system of PDEs. In three dimensions, the results presented here are obtained for the spherically symmetric case, {\it i.e.}, when $D:=B_1(0)$ is the unit  ball centered at the origin and  all the coefficients in the PDEs depend only on the  radius $r=|x|$. As a lead in to this, we consider the one-dimensional case, with $D:=(0, \,1)$. It should be noted that, the one-dimensional problem is also of interest in its own right, as the second-harmonic generation process in laser technology often is modelled as a one-dimensional problem, by considering applied waves that fall onto the nonlinear medium at normal incidence \cite{boyd}. Next we summarize our main results.

\begin{remark} \label{REM omega dependence}
\em For the sake of simplicity of presentation we are going to assume that the functions $q, \chi_1, \chi_2$ are real-valued and do not depend on the frequency $\omega$. However, all our results from Sections~\ref{SECT 1d case main}, \ref{SECT 3d case main} and \ref{SECT source f} and our methods of analysis can be adapted to the situation where $q, \chi_1, \chi_2$ and the source function $f$ from Section~\ref{SECT source f} (cf. \eqref{trans rad f}) are complex-valued functions that also depend on $\omega$. To that end additional assumptions should be imposed and some of the existing assumptions should be modified. For example, in Sections~\ref{SECT 1d case main} and \ref{SECT 3d case main} additionally we need $q, \chi_1$ and $\chi_2$ to be continuous functions of $\omega$ (near the origin in the complex plane) uniformly in $x$. Furthermore, the positivity of $\chi_2$ a.e. in $D$, as required in \eqref{M and chi2} below, should be replaced with positivity at the frequency $\omega=0$:

\begin{equation} \label{chi2>0 omega=0}
\chi_2(0, \cdot) > 0 , \qquad \text{a.e.} \ x \in D~.
\end{equation}

\n For Lemma~\ref{LEM real case} we need $\chi_2$ and $q$ to be real-valued for $\omega>0$ small enough. In Section~\ref{SECT source f} we need $q, \chi_1, \chi_2$ and $f$ to be analytic functions of $\omega$. Moreover, $q$, as a function of $\omega$ must satisfy the following symmetry relation (we suppress the $x$-dependence):

\begin{equation} \label{q symmetry}
\overline{q(\omega)} = q(-\overline{\omega}).
\end{equation}

\n The property \eqref{q symmetry} is physically justified and expresses the fact that physical fields are real. Specifically, $q$ (or rather $q-1$) is the Fourier transform of a real-valued function -- the linear sucseptibility. The nonlinear susceptibilities $\chi_1, \chi_2$ in the frequency domain also satisfy \eqref{q symmetry}, for the same reason, and in particular at $\omega=0$ they must be real-valued, which is in line with the assumption \eqref{chi2>0 omega=0}.   
\end{remark}

\subsection{About generalized transmission eigenvalues in the 1-D case} \label{SECT 1d case main}
We consider the one-dimensional problem:
\begin{equation} \label{trans 1d}
\begin{cases}
u_1'' + \omega^2 q u_1  = -  \omega^2 \chi_1 u_2 \overline{u}_1 \hspace{0.7in} &\text{in} \ D:= (0,\, 1)
\\[.05in]
u_2'' + 4\omega^2 q u_2  = -  \omega^2 \chi_2  u_1^2 &\text{in} \ D
\\[.05in]
u_2 = u_2' = 0 &\text{at} \ \partial D = \{0,1\},
\end{cases}
\end{equation}

\n\n We suppress the domain $D$ from the notation of all the introduced function spaces. In our proofs we assume that the nonlinear susceptibilities $\chi_1$ and $\chi_2$ are real valued functions of $x$ only, with

\begin{equation} \label{M and chi2}
M: = \|\chi_1\|_{L^\infty} + \|\chi_2\|_{L^\infty} < \infty, \qquad \text{and} \qquad
\chi_2(x)>0, \qquad \text{a.e.} \ x \in D.
\end{equation}

\n Further, we assume that the relative permittivity $q$ is a real valued function of $x$ only, with

\begin{equation} \label{q assumption}
q \in C^1[0,1], \qquad \qquad Q:=\|q\|_{L^\infty}>0.
\end{equation}
To define a weak solution of the problem \eqref{trans 1d}, we introduce the Hilbert space $X = H^2 \times L^2$, whose elements will be denoted using the boldface notation $\bm{u} = (u_1, u_2) \in X$. We let $\|\cdot\|_X$ and $\langle \cdot, \cdot \rangle_X$ denote the norm and the inner product of $X$, respectively, defined in the usual way

\begin{equation} \label{X inner product}
\langle \bm{u}, \bm{\phi} \rangle_X= \langle u_1,\phi_1\rangle_{H^2}+ \langle u_2, \phi_2\rangle_{L^2}, ~~ \hbox{ and } ~~\| \bm{u}\|^2_X =\langle \bm{u}, \bm{u} \rangle_X .
\end{equation}

\n Multiplying the first equation of \eqref{trans 1d} by $\overline{\phi}_2$, the second one by $\overline{\phi}_1$ and integrating by parts, we introduce the map $\CT:X \times X \to \CC$
\begin{equation} \label{T cali}
\CT(\bm{u}, \bm{\phi}) = \int_{D} \left[(u_1'' + \omega^2 q u_1)  \overline{\phi}_2  + \chi_1 \omega^2 u_2 \overline{u}_1 \overline{\phi}_2\right] dx + \int_{D} \left[u_2 ( \overline{\phi}_1'' + 4 \omega^2 q \overline{\phi}_1 )  + \chi_2 \omega^2 u_1^2  \overline{\phi}_1 \right] dx.
\end{equation}

\n In one dimension the Sobolev embedding implies that $H^2(D) \hookrightarrow C^1(\overline{D})$, with a continuous and compact injection. Therefore the integrals containing the nonlinearities in the above expression are well-defined. Moreover, for fixed $\bm{u} \in X$, the map $\bm{\phi} \mapsto \CT(\bm{u}, \bm{\phi})$ is conjugately linear and continuous as a map from $X \to X$, hence by the Riesz representation theorem, there exists an element, which we denote by $T(\bm{u}) \in X$, such that 

\begin{equation} \label{T Riesz}
\CT(\bm{u}, \bm{\phi}) = \langle T (\bm{u}), \bm{\phi} \rangle_X, \qquad \qquad  \bm{u}, \bm{\phi} \in X.
\end{equation}

\n In particular, $T : X \to X$ is a bounded (the image of a bounded set is bounded) continuous nonlinear map. Often we will write $T=T_\omega$ to explicitly indicate the dependence of $T$ on $\omega$.

\begin{definition} \label{DEF weak sol}
$(\omega, \bm{u}) \in \CC \times X$ is a solution of \eqref{trans 1d} iff $T_\omega(\bm{u}) = 0$.
\end{definition}

\begin{remark}
\normalfont
A more traditional space for the weak formulation of \eqref{trans 1d} might seem to be $L^2 \times H^2_0$, as it explicitly incorporates the zero boundary condition. In this space, however, the nonlinearity $u_1^2$ makes the resulting operator $T$ unbounded, i.e., defined only on a dense subspace of $L^2 \times H^2_0$.
\end{remark}

\vspace{.1in}

\n With these preliminaries we are ready to state our first result, which will be proven in Section~\ref{section4}:

\begin{theorem} \label{THM 1}
Assume \eqref{M and chi2} and \eqref{q assumption} hold. There exists 
a constant $c_Q(\chi_2)>0$, depending only on $Q$ and $\chi_2$, such that if

\begin{equation*}
 0 \neq \bm{u} \in X ~~ \hbox{ and } ~~0<|\omega|^2 \leq \frac{c_Q(\chi_2)}{M (1 + M \|\bm{u}\|_X)}
\end{equation*}

\n then $T_\omega(\bm{u}) \neq 0$.

\end{theorem}

\begin{remark}
\normalfont
If the nonlinear susceptibility $\chi_2$ is a positive constant, then in the above theorem we can taken $c_Q(\chi_2) = \tilde{c}_Q \chi_2$, where $ \tilde{c}_Q>0$ only depends on $Q$.
\end{remark}

\vspace{.1in}

\n
The above theorem immediately implies:

\begin{corollary} \label{CORO}
Assume \eqref{M and chi2} and \eqref{q assumption} hold. For any $r>0$ there exists $\delta = \delta(r, Q, M,\chi_2)$, such that if $0<|\omega|<\delta$ and  $0 < \|\bm{u}\|_X < r$ then 
\begin{equation*}
T_\omega(\bm{u}) \neq 0.
\end{equation*}
\end{corollary}
\n
This corollary implies that there are no generalized (and hence also no second-harmonic) transmission eigenvalues, with eigenfunctions in the ball $B_r^X(0)$, provided $\omega$ is sufficiently small and nonzero. Alternatively, Theorem~\ref{THM 1} implies that a sufficiently small $\omega$ can be a generalized transmission eigenvalue only if its corresponding eigenfunction satisfies

\begin{equation} \label{e-function blow up}
\|\bm{u}\|_X \geq \frac{c}{|\omega|^2},
\end{equation}

\n i.e., if generalized transmission eigenvalues approaching 0 exist, then their corresponding eigenfunctions blow up at a rate of $1/|\omega|^2$. This leads to the following dichotomy: either such blow-up eigenfunctions exist, or our non-existence result can be improved. Based on numerical evidence we conjecture that both situations may occur. In the degenerate case $\chi_1=0$ the two equations decouple and the variable $u_1$ appears in the second equation as a source term. In Section \ref{chi1equal0}, it is shown that the generalized transmission eigenvalues form a discrete set (without finite accumulation points) and for that reason there will be no such eigenvalues in a sufficiently small interval $(0,\delta)$. However, for $\chi_1$ strictly positive we conjecture, entirely based on numerical evidence, that the first situation occurs:

\begin{conjecture} \label{CONJ existence}
Assume \eqref{M and chi2} and \eqref{q assumption} hold. Furthermore assume $\chi_1>0$. Then there exist $\delta, c>0$ such that any $\omega \in (0, \delta)$ is a generalized transmission eigenvalue, with a corresponding eigenfunction $\bm{u}$ satisfying \eqref{e-function blow up}.    
\end{conjecture}

\n As the lemma below shows, it is important for the validity of this conjecture that the first component, $u_1$, of the eigenfunction $\bm{u}$ be complex valued:

\begin{lemma} \label{LEM real case}
Assume \eqref{M and chi2} and \eqref{q assumption} hold (in fact, we only need $q \in L^\infty(0,1)$). There exists $\delta>0$ depending only on $q$, such that if $\omega \in (0, \delta]$ and $0\neq \bm{u} =(u_1, u_2) \in X$ with  $u_1$ real  valued, then
\begin{equation*}
T_\omega(\bm{u}) \neq 0.
\end{equation*}

\n Moreover, when $q = 1$ one can take $\delta = \frac{\pi}{2}$.

\end{lemma}

\n  For the proof of this lemma we refer to Section~\ref{SECT real case}.

\subsection{About generalized transmission eigenvalues in the 3-D radially symmetric case} \label{SECT 3d case main}
Here we consider the problem 
\begin{equation} \label{trans}
\begin{cases}
\Delta u_1 + \omega^2 q u_1  = - \omega^2 \chi_1 u_2 \overline{u}_1 \hspace{0.7in} &\text{in} \ D:=\left\{|x|<1\right\}
\\[.05in]
\Delta u_2 + 4 \omega^2 q u_2  = - \omega^2 \chi_2 u_1^2 &\text{in} \ D
\\[.05in]
u_2 = \partial_\nu u_2 = 0 &\text{on} \ \partial D:=\left\{|x|=1\right\}
\end{cases}
\end{equation}
where $D:= B_1(0) \subset \RR^3$ is the unit ball centered at the origin and $q$, $\chi_1$ and $\chi_2$ are radial functions of the spacial variable, i.e., functions of $r=|x|$. Looking only for radially symmetric solutions $u_1:=y_1(r)$ and $u_2:=y_2(r)$ of  (\ref{trans}), the problem  reads
\begin{equation} \label{y tans problem}
\begin{cases}
\displaystyle y_1'' + \frac{2}{r} y_1' + \omega^2 q y_1  = - \omega^2 \chi_1 y_2 \overline{y}_1 \hspace{0.7in} &\text{for} \ r<1
\\[.1in]
\displaystyle y_2'' + \frac{2}{r} y_2' + 4 \omega^2 q y_2  = - \omega^2 \chi_2 y_1^2 &\text{for} \ r<1
\\[.1in]
y_2(1) = y_2'(1) = 0 
\\[.05in]
y_1(0), \ y_2(0) \quad \text{are bounded} ,
\end{cases}
\end{equation}
where the derivatives are with respect to $r$. 
\n We define new functions of $r$, denoted $u_1(r)$ and $u_2(r)$ (by a slight abuse of notation) as follows $u_1(r) = r y_1(r)$ and $u_2(r)=ry_2(r)$. The boundary conditions at the endpoint $r=1$ stay the same, i.e., $u_2(1) = u_2'(1) = 0$ and at the endpoint $r=0$ they become $u_1(0) = u_2(0) = 0$. Thus we end up with the following one dimensional problem for the radial functions $u_1$ and $u_2$ 
\begin{equation} \label{trans rad}
\begin{cases}
\displaystyle u_1'' + \omega^2 q u_1  = - \frac{\omega^2}{x}  \chi_1 u_2 
 \overline{u}_1 \hspace{0.7in} &\text{in} \ D:= (0,\, 1)
\\[.1in]
\displaystyle u_2'' + 4 \omega^2 q u_2  = - \frac{\omega^2}{x} \chi_2 u_1^2 &\text{in} \ D
\\[.1in]
u_2(0) = u_2(1) = u_2'(1) = 0 
\\[.05in]
u_1(0) = 0
\end{cases}
\end{equation}
\n To make it look more like the one dimensional case we have used $x$ for the radial variable. Analogously to Section~\ref{SECT 1d case main}, and again suppressing the domain $D$ from the notation of the function spaces, we let

\begin{equation} \label{X diamond}
X_{\diamond} = H^2_{\diamond} \times L^2, \qquad \qquad H^2_{\diamond} = \left\{ u \in H^2: u(0) = 0 \right\}.
\end{equation}

\n We endow the space $X_\diamond$ with the same norm and inner product as $X$ (cf. \eqref{X inner product}). Multiplying the first equation of \eqref{trans rad} by $\overline{\phi}_2$, the second one by $\overline{\phi}_1$ and integrating by parts motivates us to introduce the map $\CT^\diamond : X_\diamond \times X_\diamond \to \CC$, defined by

\begin{equation} \label{T cali rad}
\begin{split}
\CT^\diamond(\bm{u}, \bm{\phi}) =& \int_{D} \left[(u_1'' + \omega^2 q u_1)  \overline{\phi}_2  + \frac{\omega^2}{x} \chi_1 u_2 \overline{u}_1 \overline{\phi}_2 \right]dx 
\\[.05in]
&+ \int_{D} \left[u_2 \left( \overline{\phi}_1'' + 4 \omega^2 q \overline{\phi}_1 \right)  + \frac{\omega^2}{x} \chi_2 u_1^2 \,  \overline{\phi}_1 \right] dx.
\end{split}
\end{equation}

\n Using the mean value theorem and the Sobolev embedding $H^2(D) \hookrightarrow C^1(\overline{D})$ one easily concludes that there exists a constant $C>0$ such that

\begin{equation} \label{u/x bounded by H^2 norm}
\left\|\frac{u}{x}\right\|_{L^\infty} \leq C \|u\|_{H^2}, 
\qquad \qquad \forall \ u \in H^2_{\diamond}.
\end{equation}

\n In particular, the integrals containing the nonlinearities in the definition of $\CT^\diamond$ are well-defined, continuous and bounded. Thus, there exists a bounded, continuous and nonlinear map $T^\diamond : X_\diamond \to X_\diamond$ with

\begin{equation} \label{T diamond}
\CT^\diamond(\bm{u}, \bm{\phi}) = \langle T^\diamond (\bm{u}), \bm{\phi} \rangle_{X_\diamond}, \qquad \qquad  \bm{u}, \bm{\phi} \in X_\diamond.
\end{equation}

\n We shall write $T^\diamond_\omega$ to indicate the dependence of the operator $T^\diamond$ on $\omega$. We define $(\omega, \bm{u}) \in \CC \times X_\diamond$ to be a solution of \eqref{trans rad} iff $T^\diamond_\omega(\bm{u}) = 0$.

\begin{theorem}
Theorem~\ref{THM 1}, Corollary~\ref{CORO} and Lemma~\ref{LEM real case} hold for the operator $T^\diamond_\omega$ in the space $X_\diamond$.
\end{theorem}
\n
If we assume a lower bound $\chi_2(x) \geq m > 0$ for a.e. $x \in D$, then we can take $C(\chi_2) = C_0 m$, where $C_0>0$ is an absolute constant (see Remark~\ref{REM proj_0 diamond}). Consequently, the constant $\delta$ (from Corollary~\ref{CORO}) takes the form $\delta = \delta(r, Q, m, M)$.

We will present the proofs for the operator $T$ in the space $X$. All the arguments apply analogously to $T^\diamond$ in the space $X_\diamond$. In fact, in the latter case the arguments are somewhat simpler due to the fact that the linearization of $T^\diamond$ has a 1-dimensional kernel, while that of $T$ has a 2-dimensional kernel (cf. Remark~\ref{REM L diamond N and R^perp}). When appropriate, we will comment on the necessary modifications of a given argument to treat the case of $T^\diamond$.

We should point out that the conjecture \ref{CONJ existence} only applies to the one-dimensional problem (\ref{trans 1d}). 

\subsection{Solvability result in the presence of a force term} \label{SECT source f}
For this set of results, we introduce a function $g \neq 0$ in the first equation of  \eqref{trans}. We will state the result and carry out the proof in the radially symmetric case, i.e., for the problem \eqref{trans rad} and operator $T^\diamond_\omega$. An analogous result holds in the one-dimensional case. More specifically, with $f(x)=xg(x)$, we are considering the problem
\begin{equation} \label{trans rad f}
\begin{cases}
\displaystyle u_1'' + \omega^2 q u_1  = - \frac{\omega^2}{x}  \chi_1 u_2 
 \overline{u}_1 + f \hspace{0.7in} &\text{in} \  D:= (0,\, 1)
\\[.1in]
\displaystyle u_2'' + 4 \omega^2 q u_2  = - \frac{\omega^2}{x} \chi_2 u_1^2 &\text{in} \  D
\\[.1in]
u_1(0) = u_2(0) = u_2(1) = u_2'(1) = 0,
\end{cases}
\end{equation}

\n the weak formulation of which becomes

\begin{equation} \label{surj equation main}
T^\diamond_\omega(\bm{u}) = \bm{f}, \qquad \qquad \bm{f} = (0,f).
\end{equation}

\n In this section we will assume

\begin{equation} \label{chi1 chi2 q exist assum}
\chi_1, \chi_2 \in L^\infty(0,1), \qquad q \in C[0,1]  \qquad \text{and} \qquad 0 \neq f \in L^2(0,1).
\end{equation}
\n Suppose also that $q$ is real-valued. We define $z_\omega, \widetilde{z}_\omega$ as the solutions of the following initial value problems:

\begin{equation} \label{z_omega}
\begin{cases}
z_\omega'' + \omega^2 q z_\omega = 0 \quad \text{in} \ (0,1)
\\
z_\omega(0) = 0 \; \mbox{and} \; z_\omega'(0) = 1
\end{cases}
\qquad 
\begin{cases}
\widetilde{z}_\omega'' + \omega^2 q \widetilde{z}_\omega = 0 \quad \text{in} \ (0,1)
\\
\widetilde{z}_\omega(0) = 1  \; \mbox{and} \;  \widetilde{z}_\omega'(0) = 0.
\end{cases}
\end{equation}

\n Let $a_\omega$  be the solution of $a_\omega'' + \omega^2 q a_\omega = f$ with zero initial conditions $a_\omega(0) = a_\omega'(0) = 0$. More explicitly, 

\begin{equation} \label{a_omega}
a_\omega(x) = \int_0^x \big[ z_{\omega}(x) \widetilde{z}_{\omega}(t) - z_{\omega}(t) \widetilde{z}_{\omega}(x) \big] f(t) \, dt.
\end{equation}
\n We introduce the function 
\begin{equation} \label{zeta}
\zeta(\omega) = \left( \int_0^1 z_\omega  a_\omega z_{2\omega}  \, \frac{\chi_2 dx}{x} \right)^2  -  \int_0^1 z_\omega^2 z_{2\omega} \, \frac{\chi_2 dx}{x} \cdot \int_0^1 a_\omega^2 z_{2\omega} \, \frac{\chi_2 dx}{x},
\end{equation}
and the associated set 
\begin{equation} \label{Lambda}
\Lambda  = \Lambda (f, q, \chi_2) = \left\{ \omega \in \CC: \zeta(\omega) = 0  \right\} \cup \{0\}\subset \CC.
\end{equation}
\n In Lemma~\ref{LEM z_omega properties} and the discussion proceeding it is shown that this function is entire, and thus $\Lambda$ is (at most) countable, with no finite accumulation points for any $f\ne 0$. Finally, for any $\rho, \delta > 0$ (with $\rho$ large and $\delta$ small) we define
\begin{equation} \label{G set}
\CG = \CG_{\rho, \delta} = \left\{ \omega \in \CC: |\omega| \leq \rho \quad \text{and} \quad \text{dist} (\omega, \Lambda) \geq \delta \right\}
\end{equation}

\begin{theorem} \label{THM exist}
Assume \eqref{chi1 chi2 q exist assum} holds. For any $\rho, \delta > 0$ there exists $c = c(\rho, \delta, f, \chi_2, q) > 0$ such that if $\|\chi_1\|_{L^\infty} < c$, then \eqref{surj equation main} has a solution for any $\omega \in \CG_{\rho, \delta}$.
\end{theorem}

\n The proof of Theorem~\ref{THM exist} is given in Section~\ref{section5}.

\begin{remark} \mbox{}
\normalfont
We note that $\Lambda = \CC$ for $f=0$, and thus Theorem \ref{THM exist} is vacuous in that case.
\end{remark}
\n The set $\Lambda = \Lambda(f,q,\chi_2)$ may contain (countably many) real frequencies. For example, let $q = 1$, then
\begin{equation*}
z_\omega(x) = \frac{\sin(\omega x)}{\omega}, \qquad \qquad
\widetilde{z}_{\omega}(x) = \cos(\omega x).
\end{equation*}
\n Further, let us also take $\chi_2 = 1$ and $f(x)=x$. Figure~1 below shows a plot of the function $\omega^{10} \zeta(\omega)$ given these choices.
\begin{figure}[H]
\centering
\captionsetup{width=.84\linewidth}
\includegraphics[scale=0.5]{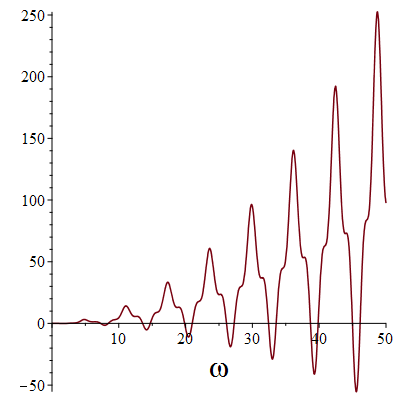}
\caption{The plot of $\omega^{10} \zeta(\omega)$ for $q=\chi_2 =1$ and $f=x$.}
\end{figure}
\subsection{Generalized transmission eigenvalues in the degenerate case $\chi_1=0$}
\label{chi1equal0}
Consider the problem \eqref{trans rad} in the degenerate case $\chi_1=0$. Further, let us assume that $\chi_2 \in L^\infty (0,1)$ is a nonzero function and $q \in C[0,1]$. Notice that in this case the system \eqref{trans rad} decouples and it is straightforward to show that $0 \neq \omega \in \CC$ is an eigenvalue of \eqref{trans rad} iff it is a zero of the following entire function (see Lemma~\ref{LEM z_omega properties} for the proof that this function is entire):

\begin{equation*}
\Upsilon_1(\omega) = \int_0^1 z_\omega^2 z_{2\omega} \frac{\chi_2 dx}{x},
\end{equation*} 

\n where $z_\omega$ is defined by \eqref{z_omega}. Consequently, the eigenvalues form a discrete set in the complex plane (without finite accumulation points). Let us now choose $q=1$, and take $\chi_2 = \mathbbm{1}_{(\frac{1}{2}, 1)}$, i.e., the material that occupies the unit ball $B_1 \subset \RR^3$ is extremely  inhomogeneous, its core $B_{\frac{1}{2}}$ is linear and the nonlinearity is supported near the boundary: in the annulus $B_1 \backslash B_{\frac{1}{2}}$. With these choices there are (countable many) real eigenvalues. Indeed, Figure~\ref{FIG2a} below, shows that $\Upsilon_1$ has (countably many) positive zeros. We mention that the same conclusion remains true as long as $\chi_2$ is supported away from the origin, i.e., for $\chi_2 = \mathbbm{1}_{(c, 1)}$ for any $0<c<1$. However, with the choice $\chi_2 = 1$, the function $\Upsilon_1$ appears to have no positive zeros.

\begin{figure}[h]
\begin{center}
\begin{subfigure}[t]{2.5in}
\includegraphics[scale=0.5]{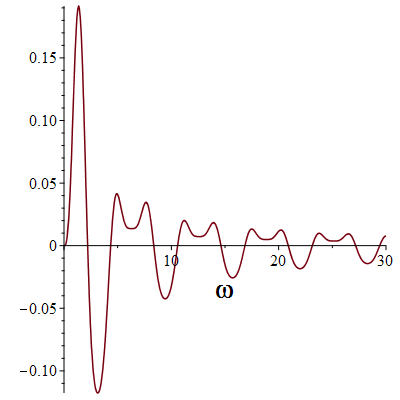}
\caption{The plot of $\omega^3 \Upsilon_1(\omega)$ for 
$q=1$
\\
and $\chi_2 = \mathbbm{1}_{(\frac{1}{2}, 1)}$.}
\label{FIG2a}
\end{subfigure}
\hspace{1in}
\begin{subfigure}[t]{2.5in}
\includegraphics[scale=0.5]{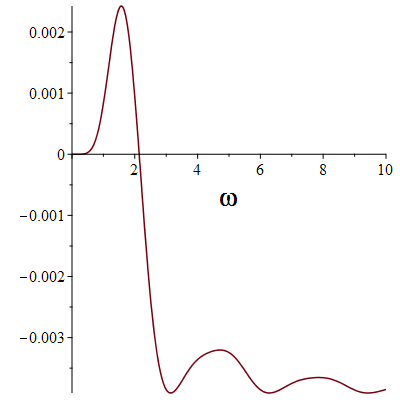}
\caption{The plot of $\omega^6 \Upsilon_2(\omega)$ for $q=1$
\\
and $\chi_2 = 1$.}
\label{FIG2b}
\end{subfigure}
\caption{}
\end{center}
\end{figure}

\n In the 1d case, i.e., for the problem \eqref{trans 1d} the situation is interestingly different. Substituting $u_1 = \alpha \, z_\omega + \beta \, \widetilde{z}_\omega$ into the equation for $u_2$, we find that $0\neq \omega \in \CC$ is an eigenvalue of \eqref{trans 1d} iff

\begin{equation} \label{1d chi1=0 existence cond}
\begin{cases}
a_1 \alpha^2 + b_1 \alpha \beta + c_1 \beta^2 = 0,
\\
a_2 \alpha^2 + b_2 \alpha \beta + c_2 \beta^2 = 0,
\end{cases}
\qquad \text{for some} \ (\alpha, \beta) \in \CC^2 \backslash \{0\},
\end{equation}

\n where

\begin{equation*}
\begin{cases}
\displaystyle a_1(\omega) = \int_0^1 z_\omega^2 \, z_{2\omega} \, \chi_2 dx, 
\\[.15in]
\displaystyle b_1(\omega) = 2 \int_0^1 z_\omega \, \widetilde{z}_\omega \, z_{2\omega}  \chi_2 dx, 
\\[.15in]
\displaystyle c_1(\omega) = \int_0^1 \widetilde{z}_\omega^2 \, z_{2\omega} \chi_2 dx
\end{cases}
\qquad \text{and} \qquad
\begin{cases}
\displaystyle a_2(\omega) = \int_0^1 z_\omega^2 \, \widetilde{z}_{2\omega} \chi_2 dx, 
\\[.15in]
\displaystyle b_2(\omega) = 2 \int_0^1 z_\omega \, \widetilde{z}_\omega \, \widetilde{z}_{2\omega} \chi_2 dx, 
\\[.15in]
\displaystyle c_2(\omega) = \int_0^1 \widetilde{z}_\omega^2 \, \widetilde{z}_{2\omega} \chi_2 dx.    
\end{cases}
\end{equation*}

\n It is then straightforward to show that \eqref{1d chi1=0 existence cond} is equivalent to $\Upsilon_2 (\omega) = 0$, where

\begin{equation*}
\Upsilon_2 = \left( a_1 c_2 - a_2 c_1 \right)^2 - (a_1 b_2 - a_2 b_1) (b_1 c_2 - b_2 c_1).
\end{equation*}

\n Discreteness of eigenvalues again follows from the fact that $\Upsilon_2$ is an entire function. Let us now take $q=1$. Figure~\ref{FIG2b} shows that $\Upsilon_2$ has one positive zero for $\chi_2 = 1$. Further numerical plots show that for the choice $\chi_2 = \mathbbm{1}_{(c, 1)}$ (with $0<c<1$ arbitrary) the function $\Upsilon_2$ has a similar shape as that in Figure~\ref{FIG2b}, i.e., it continues to have only one positive zero.

\section{Proof of Lemma~\ref{LEM real case}} \label{SECT real case}
\setcounter{equation}{0}

The proof of Lemma~\ref{LEM real case} is based on the following lower bound for the nonlinear operator $T$ \eqref{T Riesz}: for any $\bm{u}=(u_1, u_2) \in X$ and $\omega \in \CC$

\begin{equation} \label{T lower bound}
\|T_\omega(\bm{u})\|_X \geq \frac{|\omega|^2}{\|z_{2\omega}\|_{H^2}} \, \left| \int_0^1 \chi_2 u_1^2 \ z_{2\omega} \, dx \right|,
\end{equation}

\n where $z_\omega$ is defined by \eqref{z_omega}. To prove this bound we use that

\begin{equation*}
\|T(\bm{u})\|_X = \sup_{\|\bm{\phi}\|_X=1} \left| \langle T(\bm{u}), \bm{\phi} \rangle_X  \right|.
\end{equation*}

\n Now, make the choice

\begin{equation*}
\bm{\phi} = (\phi_1, 0), \qquad \qquad \phi_1 = \frac{z_{-2\overline{\omega}}}{\|z_{-2\overline{\omega}}\|_{H^2}} 
\end{equation*}

\n and use the fact that $\overline{\phi}_1'' + 4 \omega^2 q \overline{\phi}_1 = 0$. Consequently,

\begin{equation*}
\langle T(\bm{u}), \bm{\phi} \rangle_X = \frac{\omega^2}{\|z_{-2\overline{\omega}}\|_{H^2}} \,  \int_0^1 \chi_2 u_1^2 \ \overline{z_{-2\overline{\omega}}} \, dx. 
\end{equation*}

\n  To conclude the proof of \eqref{T lower bound}, it remains to note that

\begin{equation} \label{z sym}
\overline{z_{-2\overline{\omega}}} = z_{2\omega},
\end{equation}

\n which directly follows from the definition of $z_\omega$ \eqref{z_omega} and the fact that $q$ is real-valued (or more generally, when $q$ is also $\omega$-dependent and complex-valued, we use the symmetry assumption \eqref{q symmetry}). In Lemma~\ref{LEM y} of the Appendix, we show that as $\omega \to 0$

\begin{equation*}
z_\omega(x) = x + o(1), \qquad \qquad z_\omega'(x) = 1 + o(1),
\end{equation*}
\n uniformly in $x$. In particular, for real $\omega$  there exists $\delta = \delta(q) > 0$ such that $z_{2\omega}(x) >0$ for a.e. $x \in (0,1)$ and $0\leq \omega \leq \delta$.

Suppose now, for the sake of a contradiction, that Lemma~\ref{LEM real case} is not true, i.e. , $T_\omega(\bm{u}) = 0$ for some $\omega \in (0, \delta]$ and $\bm{u} \neq 0$ with $u_1$ real valued. Then \eqref{T lower bound} implies

\begin{equation*}
\int_0^1 \chi_2 u_1^2 \, z_{2\omega} dx = 0.
\end{equation*}

\n As all the functions inside the integral are non-negative, we conclude that the integrand is zero and hence $u_1 = 0$ a.e. in $(0,1)$. From the definition of $T$ we then immediately also find that $u_2= 0$. In combination $\bm{u} = 0$, which is a contradiction.

\begin{remark}
\normalfont When $q=1$ we have $z_{2\omega} = \sin(2\omega x)/2\omega$ and this function is positive a.e. in $(0,1)$ provided $\omega \leq \frac{\pi}{2}$. Consequently, we may take $\delta = \frac{\pi}{2}$.
\end{remark}

\bigskip 
In the next section we discuss the Fredholm property of the linearization of the nonlinear system, which is an intermediate tool needed throughout our remaining proofs. 

\section{The Fredholm property of linearization} \label{SECT weak form and Fredholm}
\setcounter{equation}{0}

In this paper we use the following norm and inner product on the space $H^2$

\begin{equation} \label{H^2 norm}
\|u\|_{H^2}^2 = \|u\|_{L^2}^2 + \|u''\|_{L^2}^2,
\qquad \qquad
\langle u,v \rangle_{H^2} = \langle u,v \rangle_{L^2} + \langle u'', v'' \rangle_{L^2}.
\end{equation}

\n This norm is equivalent to the standard $H^2$-norm involving $u'$ as well. Let $T$ be the operator from \eqref{T Riesz} and consider its linear part $L=L_\omega : X \to X$ defined by

\begin{equation} \label{L}
\langle L (\bm{u}), \bm{\phi} \rangle_X = 
\int_{D} \left( u_1'' + \omega^2 q u_1 \right) \overline{\phi}_2 dx + \int_{D} u_2 \left( \overline{\phi}_1'' + 4 \omega^2 q \overline{\phi}_1 \right) dx.
\end{equation}

\n The operator $L$ is in fact the Frech{\'e}t derivative of $T$ at $\bm{u} = 0$. Indeed, $T(0) = 0$ and as $\bm{u} \to 0$ in $X$ we have

\begin{equation*}
T(\bm{u}) = L(\bm{u}) + o(\|\bm{u}\|_X), 
\end{equation*}

\n consequently

\begin{equation*}
\partial_{\bm{u}} T \rvert_{\bm{u}=0} = L.
\end{equation*}

\n However, we mention that $T$ is not Frech{\'e}t differentiable at any $\bm{u} \neq 0$, because one of the nonlinearities contains complex conjugation. We now state the main result of this section:

\begin{theorem} \label{THM L}
Let $\omega \in \CC$ and $L = L_\omega: X \to X$ be given by \eqref{L}. Assume $q$ satisfies \eqref{q assumption}. Then, $L$ is a Fredholm operator of index 0, its kernel and the orthogonal complement of its range in $X$ are $2$-dimensional spaces given by

\begin{equation} \label{N and R^perp L}
\begin{split}
N_\omega:=N(L_\omega) &= \left\{ (u, 0) : u \in H^2 \quad \text{and} \quad u'' + \omega^2 q u = 0 \right\}
\\[.05in]
R_\omega^\perp :=R(L_\omega)^\perp &= \left\{ (u, 0) : u \in H^2 \quad \text{and} \quad  u'' + 4 \overline{\omega}^2 q u = 0 \right\} 
\end{split}
\end{equation}
  
\end{theorem}

\vspace{.1in}

\begin{remark} \label{REM L diamond N and R^perp}
\normalfont
For the operator $T^\diamond: X_\diamond \to X_\diamond$ \eqref{T diamond}, its linearization $L^\diamond: X_\diamond \to X_\diamond$ is given by the same formula \eqref{L}. Furthermore, the above theorem still holds for $L^\diamond$ with $X_\diamond$ in place of $X$, and with $H^2$ replaced by $H^2_\diamond$ in the formula \eqref{N and R^perp L}. In particular, the kernel $N^\diamond_\omega$ and the orthogonal complement of the range $(R^\diamond_\omega)^\perp$ of $L^\diamond$ are now 1-dimensional. 
\end{remark}

\vspace{.2in}

Let us start by verifying \eqref{N and R^perp L}. If $L(\bm{u}) = 0$, using the definition \eqref{L} and choosing $\phi_1 = 0$ and $\phi_2 = u_1'' + \omega^2 q u_1 \in L^2$, we conclude that $u_1'' + \omega^2 q u_1 = 0$ in $D=(0,1)$. Similarly, choosing $\phi_2=0$ and $\phi_1 \in H^2$ to be a solution of $\phi_1'' + 4 \overline{\omega}^2 q \phi_1 = u_2$, we conclude that $u_2=0$. This proves the inclusion $\subset$ of the first formula of \eqref{N and R^perp L}. The other inclusion is trivial. To prove the second formula we just observe  that the adjoint $L^*:X \to X$ is given by

\begin{equation} \label{L*}
\langle L^* (\bm{u}), \bm{\phi} \rangle_X = 
\int_{D} \left( u_1'' + 4 \overline{\omega}^2 q u_1 \right) \overline{\phi}_2 dx + \int_{D} u_2 \left( \overline{\phi}_1'' + \overline{\omega}^2 q \overline{\phi}_1 \right) dx,
\end{equation}

\n and that

\begin{equation*}
R(L)^\perp = N(L^*).
\end{equation*}

To prove the Fredholm property, let us write

\begin{equation*}
L = L_0 + \omega^2 K,
\end{equation*}

\n where the operators $L_0, K : X \to X$ are given by 

\begin{equation} \label{L0 and K}
\langle L_0 (\bm{u}), \bm{\phi} \rangle_X = 
\int_{D} \left(  u_1'' \overline{\phi}_2 + u_2 \overline{\phi}_1'' \right) dx,
\qquad \qquad
\langle K (\bm{u}), \bm{\phi} \rangle_X = 
\int_{D} \left( q u_1 \overline{\phi}_2 + 4 q u_2 \overline{\phi}_1 \right) dx.
\end{equation}

\n In the two lemmas below we show that $L_0$ is a Fredholm operator with index 0 and $K$ is a compact operator. Since a compact perturbation of a Fredholm operator with index 0 stays Fredholm with index 0 \cite{kato95} this concludes the proof of Theorem~\ref{THM L}.

\begin{lemma} \label{LEM L0 Fredholm}
$L_0: X \to X$ given by \eqref{L0 and K} is a self-adjoint, Fredholm operator with index 0.
\end{lemma}

\begin{proof}
The self-adjointness is obvious. Let $N_0$ and $R_0$ denote the kernel and the range of the operator $L_0$, respectively. We will show that the range $R_0 \subset X$ is closed. This will imply that $X = R_0 \oplus R_0^\perp$. But, from \eqref{N and R^perp L} with $\omega = 0$ we know that $N_0$ and $R_0^\perp$ are both 2-dimensional, therefore $L_0$ is Fredholm with index 0. 

To prove that the range is closed, let us split $X$ into the direct sum $X = N_0 \oplus N_0^\perp$. It is enough to show that there exists $c>0$, such that

\begin{equation} \label{lower bound L0}
\|L_0(\bm{u})\|_X \geq c \|\bm{u}\|_X, \qquad \qquad \forall \ \bm{u} \in N_0^\perp.
\end{equation}

\n Indeed, suppose $\{\bm{u}^n\} \subset X $ is such that $L_0(\bm{u}^n) \to \bm{y}$ in $X$ for some $\bm{y} \in X$. Let us split $\bm{u}^n = \bm{v}^n + \bm{w}^n$, where $\bm{v}^n \in N_0$ and $\bm{w}^n \in N_0^\perp$. Then $L_0(\bm{u}^n) = L_0(\bm{w}^n) \to \bm{y}$. Using \eqref{lower bound L0} we get

\begin{equation*}
\|L_0(\bm{w}^n) - L_0(\bm{w}^m)\|_X \geq c \|\bm{w}^n - \bm{w}^m\|_X, 
\end{equation*}

\n which shows that $\{\bm{w}^n\}$ is Cauchy and hence convergent: $\bm{w}^n \to \bm{w}$ for some $\bm{w} \in X$. Thus, we conclude that $\bm{y}$ is in the range of $L_0$, as $\bm{y} = L_0(\bm{w})$.

To prove \eqref{lower bound L0}, assume that this inequality is not true, then there exists a sequence $\{\bm{u}^n\} \subset N_0^\perp$ with $\|\bm{u}^n\|_X = 1$, such that $L_0(\bm{u}^n) \to 0$ in $X$. Extract a weakly convergent subsequence, which we do not relabel, such that $\bm{u}^n \rightharpoonup \bm{u}$, for some $\bm{u} = (u_1, u_2) \in X$. We also write $\bm{u}^n = (u_1^n, u_2^n)$. Specifically,

\begin{equation*}
u_1^n \rightharpoonup u_1 \quad \text{in} \quad H^2
\qquad \qquad \text{and} \qquad \qquad
u_2^n \rightharpoonup u_2 \quad \text{in} \quad L^2.
\end{equation*}

\n For any $\bm{\phi} \in X$, we can use the definition of weak convergence and \eqref{L0 and K} to conclude that  

\begin{equation*}
0 = \lim_{n \to \infty} \langle L_0 (\bm{u}^n), \bm{\phi} \rangle_X  = \langle L_0 (\bm{u}), \bm{\phi} \rangle_X.
\end{equation*}

\n In other words $L_0 (\bm{u}) = 0$, or equivalently, $\bm{u} \in N_0$. On the other hand $N_0^\perp$ is weakly closed, hence $\bm{u} \in N_0^\perp$. Thus, we obtain that $\bm{u} = 0$. To arrive at a contradiction we are going to show that, in fact, we have strong convergence $\bm{u}^n \to 0$ in $X$. 

The convergence $L_0(\bm{u}^n) \to 0$ implies that

\begin{equation*}
\sup_{\|\bm{\phi}\|_X \leq 1} \left| \langle L_0 (\bm{u}^n), \bm{\phi} \rangle_X  \right| \longrightarrow 0.
\end{equation*}

\n In particular, writing $\bm{\phi} = (\phi, \psi)$ and using the definition of $L_0$ we obtain

\begin{equation} \label{convergence sup}
\sup_{\|\phi\|_{H^2} \leq 1} \left| \int_{D} u_2^n \ \overline{\phi}'' dx \right| \longrightarrow 0,
\qquad \qquad
\sup_{\|\psi\|_{L^2} \leq 1} \left| \int_{D} (u_1^n)'' \ \overline{\psi} dx \right| \longrightarrow 0.
\end{equation}

\n The second relation of \eqref{convergence sup} clearly implies that $(u_1^n)'' \to 0$ in $L^2$. Consequently, $u_1^n \to 0$ in $H^2$ (since we already know that $u_1^n \to 0$ in $L^2$, due to the weak convergence in $H^2$).

It remains to prove $u_2^n \to 0$ in $L^2$. To that end let $\phi_n \in H^2$ be such that $\phi_{n}''= u^n_2$. Specifically, let us take 

\begin{equation*}
\phi_{n}(x) = \int_0^x \int_0^t u^n_2(s) ds dt, 
\end{equation*}

\n then it is clear that for some constant $C>0$, that is independent of $n$,

\begin{equation*}
\|\phi_{n}\|_{H^2} \leq C \|u^n_2\|_{L^2}.
\end{equation*}

\n Since $\{u^n_2\} \subset L^2$ is bounded, we conclude that $\{\phi_{n}\} \subset H^2$ is also bounded and, using the first equation of \eqref{convergence sup} we deduce that $u^n_2 \to 0$ in $L^2$.  
\end{proof}

\begin{lemma} \label{LEM K compact}
Assume $q$ satisfies \eqref{q assumption}, then $K: X \to X$ given by \eqref{L0 and K} is a compact operator.
\end{lemma}

\begin{proof}
Let us show that $K$ maps a weakly convergent sequence to a strongly convergent one. $K$ is a sum of two terms, compactness of the first term is obvious, as $u_1^n \rightharpoonup 0$ in $H^2$ implies strong convergence in $L^2$. Let us prove the compactness of the second term, so let $u_n:=u_2^n \rightharpoonup 0$ in $L^2$. In particular, there exists a constant $c_0>0$, such that $\|u_n\|_{L^2} \leq c_0$. We need to prove

\begin{equation} \label{sup to 0}
\sup_{\|\phi\|_{H^2} \leq 1} \left| \int_{D} q u_n \overline{\phi} dx \right| \longrightarrow 0.
\end{equation}

\n Consider the average of $u_n$ over $D = (0,1)$, i.e.,

\begin{equation*}
a_n = \int_D u_n(x) dx.
\end{equation*}

\n The weak convergence shows that $a_n \to 0$. Let us now introduce the function

$$v_n(x) = \int_0^x \int_0^t \left(u_n(s) - a_n \right)ds dt.$$

\n Note that $v_n$ solves the following Neumann problem:

\begin{equation*}
\begin{cases}
v_n'' = u_n -a_n
\\
v_n'(0) = v_n'(1) = 0.
\end{cases}
\end{equation*}

\n Therefore, integrating by parts and using the boundary conditions we can write

\begin{equation*}
\int_{D} q u_n \overline{\phi} dx = \int_{D} q (v_n'' + a_n) \overline{\phi} dx = \int_{D} v_n' \left(q \, \overline{\phi}' + q' \, \overline{\phi} \right) dx + a_n \int_D q \overline{\phi} dx .
\end{equation*}

\n Consequently, for some constant $C>0$, that depends on $L^\infty$-norms of $q$ and $q'$, we get

\begin{equation*}
\sup_{\|\phi\|_{H^2} \leq 1} \left| \int_{D} q u_n \overline{\phi} dx \right| \leq C \left( \|v_n'\|_{L^2} + |a_n| \right).
\end{equation*}

\n It remains to show that $v_n' \to 0$ in $L^2$. Note that

\begin{equation*}
v_n'(x) = \int_0^x u_n(s) ds - a_n x.
\end{equation*}

\n The weak convergence $u_n \rightharpoonup 0$ in $L^2$ implies the pointwise convergence $v_n'(x) \to 0$ for any $x\in D$. Furthermore, it is clear that $|v_n'(x)| \leq 2 c_0$ for all $x \in D$. We may now use  Lebesgue's dominated convergence result to conclude the proof.
\end{proof}

\section{Proof of Theorem~\ref{THM 1} } \label{section4}
\setcounter{equation}{0}

This proof is separated into three subsections.

\subsection{Lyapunov's projection equation} \label{SECT Lyap proj}

Let $\omega \in \CC$ and $T$ be the operator given by \eqref{T Riesz}. Let us write it as

\begin{equation*}
T = L + \omega^2 A,
\end{equation*}

\n where $L$ given by \eqref{L} is the linear part of $T$ and $A:X\to X$ is the nonlinear part:

\begin{equation} \label{A}
\langle A (\bm{u}), \bm{\phi} \rangle_X = 
\int_{D} \left( \chi_1 u_2 \overline{u}_1 \overline{\phi}_2 + \chi_2 u_1^2  \overline{\phi}_1 \right) dx.
\end{equation}

\n Let $N, R$ be the kernel and the range of $L$, respectively. We note that $L$, and consequently, its kernel and range depend on $\omega$. Mostly, we will suppress this dependence for the ease of notation, but occasionally this dependence will be shown explicitly with subscript $\omega$. Recall that $L$ has closed range (cf. Theorem~\ref{THM L}). Therefore, we can decompose

\begin{equation*}
X = N \oplus N^\perp, \qquad \qquad X = R \oplus R^\perp. 
\end{equation*}

\n Let $P_{N}: X \to N$ denote the orthogonal projection onto $N$, we define $P_R$ analogously. Solving $T(\bm{u}) = 0$ is equivalent to solving the equations

\begin{equation} \label{projection eqs Lyap}
\begin{cases}
P_{R} T(\bm{u}) = 0
\\
P_{R^\perp} T(\bm{u}) = 0
\end{cases}
\qquad \Longleftrightarrow \qquad
\begin{cases}
L(\bm{u}) + \omega^2 P_{R} A(\bm{u}) = 0
\\
\omega^2 P_{R^\perp} A(\bm{u}) = 0,
\end{cases}
\end{equation}

\n where we used that $P_{R}L =L$ and $P_{R^\perp}L = 0$. Writing

\begin{equation*}
\bm{u} = \bm{v} + \bm{w}, \qquad \bm{v} \in N, \ \bm{w} \in N^\perp 
\end{equation*}

\n the first equation of \eqref{projection eqs Lyap} can be rewritten as a fixed point equation for $\bm{w}$:

\begin{equation} \label{w fixed point eq and B}
\bm{w} = \CB(\bm{v} + \bm{w}), 
\qquad \hbox{ with }\qquad \CB(\bm{u}) = - \omega^2 L^{-1} P_{R} A(\bm{u})\in N^\perp .
\end{equation} 

\n Here we used that $L : N^\perp \to R$ is invertible with bounded inverse (Theorem~\ref{THM L}). The main idea of Lyapunov's projection method (cf. \cite{zeidler_NFA_I}) is to solve \eqref{w fixed point eq and B}, assuming $\CB(\bm{v}+ \cdot)$ is a contraction on $N^\perp$, and substitute the solution $\bm{w}(\bm{v})$ into the second equation:

$$\omega^2 P_{R^\perp} A(\bm{v} + \bm{w}(\bm{v})) = 0.$$

\n As $R^\perp$ is 2-dimensional (see \eqref{N and R^perp L}), we obtain that the equation $T(\bm{u})=0$ can be rewritten as a system of 2 nonlinear equations, with 2 unknowns (the unknown $\bm{v}$ lies in $N$, which also is 2-dimensional), see \eqref{projection eq} below.  

Thus, we need to analyze the properties of the operator $\CB$. To that end, let us first derive basic estimates for $A$.

\begin{lemma} \label{LEM A Lip}
Let $A: X \to X$ be given by \eqref{A} and $M>0$ be defined by \eqref{M and chi2}. There exists an absolute constant $C>0$, such that for any $\bm{u}, \bm{v} \in X$

\begin{equation} \label{A Lip}
\begin{split}
\|A(\bm{u})\|_X &\leq C M \|\bm{u}\|^2_X
\\[.03in]
\|A(\bm{u}) - A(\bm{v})\|_X &\leq C M \max\{\|\bm{u}\|_X, \|\bm{v}\|_X\} \|\bm{u} - \bm{v}\|_X.
\end{split}
\end{equation}
 
\end{lemma}

\begin{proof}
The first inequality of \eqref{A Lip} is obvious, let us show the second one. Note that

\begin{equation*}
\langle A (\bm{u})- A (\bm{v}), \bm{\phi}  \rangle_X = \int_{D} \left[\chi_1  \overline{u}_1 (u_2-v_2) \overline{\phi}_2 + \chi_1 v_2 \left( \overline{u}_1 - \overline{v}_1 \right) \overline{\phi}_2 + \chi_2 (u_1+v_1) (u_1-v_1) \overline{\phi}_1 \right]\, dx.
\end{equation*}

\n Bounding $\chi_1, \chi_2$ by $M$, using H{\"o}lder's inequality and taking supremum over $\|\bm{\phi}\|_X \leq 1$ we arrive at

\begin{equation*}
\begin{split}
\frac{1}{M} \|A(\bm{u})-A(\bm{v})\|_X \leq & \|u_1\|_{L^\infty} \ \|u_2-v_2\|_{L^2} + \|v_2\|_{L^2} \ \|u_1-v_1\|_{L^\infty} + 
\\[.01in]
&+\|u_1 +v_1\|_{L^\infty} \ \|u_1-v_1\|_{L^2}.
\end{split}
\end{equation*}

\n The desired estimate follows, since the $L^\infty$-norms in the above inequality are bounded by the $H^2$-norms due to the Sobolev embedding theorem.
\end{proof}

We introduce the following notation (signifying its $\omega$-dependence)

\begin{equation} \label{gamma}
\gamma_\omega = \left\| L_\omega^{-1} \right\|_{\CL(R_\omega, N_\omega^\perp)},
\end{equation}

\n and use $B^{N^\perp}_r(0)$ to denote the open ball of radius $r>0$, centered at the origin in the Hilbert space $N^\perp$.

\begin{lemma} \label{LEM B}
Let $M, \gamma_\omega$ be given by \eqref{M and chi2}, \eqref{gamma}, respectively, and let $C$ be the absolute constant from Lemma~\ref{LEM A Lip}. For any $r>0$ with

\begin{equation} \label{contraction cond}
C M |\omega|^2 \gamma_\omega  r  < \frac{1}{4},
\end{equation} 

\n and any fixed $\bm{v} \in N$ with $\|\bm{v}\|_X\leq r$   

\begin{enumerate}
\item[(i)] the map $\CB(\bm{v} + \cdot) : N^\perp \to N^\perp$ is a contraction that maps the closed ball $\overline{B_r^{N^\perp}}(0)$ into itself. In particular, the equation $\bm{w}=\CB(\bm{v} + \bm{w})$ has a unique solution $\bm{w}(\bm{v})$ with $\|\bm{w}(\bm{v})\|_X \leq r$. 

\item[(ii)] furthermore, we have the estimate

\begin{equation} \label{w(v) v^2 bound}
\|\bm{w}(\bm{v})\|_X \leq 4C M |\omega|^2 \gamma_\omega \|\bm{v}\|_X^2,
\end{equation}
and in particular $\bm{w}(0)=0$.

\end{enumerate}
\end{lemma}

\begin{proof}
Due to Lemma~\ref{LEM A Lip}, for any $\bm{v} \in N$ and $\bm{w}_1, \bm{w}_2 \in N^\perp$

\begin{equation} \label{B Lip}
\begin{split}
\|\CB(\bm{v})\|_X &\leq C M |\omega|^2 \gamma_\omega \|\bm{v}\|_X^2
\\[.05in]
\|\CB(\bm{v} + \bm{w}_1) - \CB(\bm{v} + \bm{w}_2)\|_X &\leq 2 C M |\omega|^2 \gamma_\omega \max \left\{ \|\bm{v}\|_X, \|\bm{w}_j\|_X \right\} \ \|\bm{w}_1 - \bm{w}_2\|_X
\end{split}
\end{equation}

\n Suppose now that $\bm{v}, \bm{w}_j$ have norms less than or equal to $r$ and assume \eqref{contraction cond} holds, then

\begin{equation*}
\|\CB(\bm{v} + \bm{w}_1) - \CB(\bm{v} + \bm{w}_2)\|_X < \frac{1}{2} \ \|\bm{w}_1 - \bm{w}_2\|_X,
\end{equation*}

\n proving that $\CB(\bm{v} + \cdot)$ is a contraction. A combination of the last estimate with the first inequality of \eqref{B Lip} gives

\begin{equation*}
\|\CB(\bm{v} + \bm{w}_1)\|_X \leq \|\CB(\bm{v})\|_X + \frac{1}{2} \ \|\bm{w}_1 \|_X < \frac{r}{4} + \frac{r}{2} < r,
\end{equation*}

\n showing that $\CB(\bm{v} + \cdot)$ maps the closed ball $\overline{B_r^{N^\perp}}(0)$ into itself. This finishes the proof of part $(i)$, as by the Banach fixed point theorem, the equation $\bm{w}=\CB(\bm{v} + \bm{w})$ has a unique fixed point $\bm{w}(\bm{v}) \in \overline{B_r^{N^\perp}}(0)$. 

To conclude the proof it remains to establish \eqref{w(v) v^2 bound}. Let us write $\bm{w}=\bm{w}(\bm{v})$. From the fixed point equation and Lemma~\ref{LEM A Lip} we estimate

\begin{equation*} 
\|\bm{w}\|_X \leq C M |\omega|^2 \gamma_\omega \|\bm{v} + \bm{w}\|_X^2 
\leq
2C M |\omega|^2 \gamma_\omega \|\bm{v}\|_X^2 + 2C M |\omega|^2 \gamma_\omega \|\bm{w}\|_X^2.
\end{equation*}

\n The second term in the last sum  can be estimated as follows: 

\begin{equation*} 
2C M |\omega|^2 \gamma_\omega \|\bm{w}\|_X^2 \leq 2C M |\omega|^2 \gamma_\omega r \|\bm{w}\|_X \leq \frac{1}{2} \|\bm{w}\|_X,  
\end{equation*}

\n where in the first step we used that $\|\bm{w}\|_X \leq r$, and in the second one we used the hypothesis \eqref{contraction cond}. The last two inequalities clearly finish the proof.

\end{proof}

\vspace{.1in}

Combining the above lemma with \eqref{projection eqs Lyap} we arrive at:

\begin{proposition} \label{THM projection}
Let $r>0$ and suppose \eqref{contraction cond} holds. If $\bm{u} \in X$ with $\|\bm{u}\|_X \leq r$ and 

\begin{equation}
\bm{u} = \bm{v} + \bm{w}, \qquad \qquad \bm{v} \in N, \ \bm{w} \in N^\perp, 
\end{equation}

\n then, 

\begin{equation} \label{projection eq}
T(\bm{u}) = 0
\qquad \Longleftrightarrow \qquad
\begin{cases}
\bm{w} = \bm{w}(\bm{v})
\\[.02in]
\omega^2 P_{R^\perp} A(\bm{u}) = 0,
\end{cases}
\end{equation}

\n where $\bm{w}(\bm{v})$ is the unique solution of the equation \eqref{w fixed point eq and B} with $\| \bm{w}(\bm{v}) \|_X\leq r$.
\end{proposition}

\subsection{The main idea of the proof and some auxiliary lemmas} \label{SECT main idea of nonexistence}

Let $\omega \in \CC$ and $N=N_\omega$ be given by \eqref{N and R^perp L}. Let $z_\omega, \widetilde{z}_\omega$ be defined by \eqref{z_omega}. These functions form a basis for $N_\omega$. Therefore, a generic element of $N_\omega$ has the form 

\begin{equation} \label{v_omega basis expansion}
\bm{v}_{\omega,\bm{\alpha}} = \left( \alpha_1 z_\omega + \alpha_2 \widetilde{z}_\omega, 0 \right),
\qquad \qquad \bm{\alpha}=(\alpha_1, \alpha_2) \in \CC^2.
\end{equation}

\n According to Proposition~\ref{THM projection}, searching for a nontrivial solution to $T(\bm{u})=0$ is (provided $\omega \neq 0$ and the hypothesis \eqref{contraction cond} holds) equivalent to solving 

\begin{equation} \label{proj omega}
\begin{cases}
P_{R_{\omega}^\perp} A \left( \bm{v}_{\omega,\bm{\alpha}} + \bm{w} (\bm{v}_{\omega,\bm{\alpha}})  \right)  = 0   
\\
\bm{\alpha} \in \CC^2 \backslash \{0\},
\end{cases}
\end{equation}

\n where the unknown is the constant vector $\bm{\alpha}=(\alpha_1, \alpha_2)$.

The main idea of the proof of Theorem~\ref{THM 1} is to first consider the problem \eqref{proj omega}, corresponding to $\omega = 0$. Formally $\bm{w}=0$ in that case, and we therefore consider the problem

\begin{equation} \label{proj 0}
\begin{cases}
P_{R_{0}^\perp} A \left( \bm{v}_{0,\bm{\alpha}} \right)  = 0   
\\
\bm{\alpha} \in \CC^2\backslash \{0\}.
\end{cases}
\end{equation}

\n It will be shown that \eqref{proj 0} has no solution and a lower bound on $P_{R_{0}^\perp} A(\bm{v})$ will be obtained. Then, we will derive closeness estimates between the problems \eqref{proj 0} and \eqref{proj omega}. Carefully tracking the $\omega$-dependence in these estimates it will be shown that under an appropriate smallness assumption on $\omega$, the problem \eqref{proj omega} will have no solution either. We mention that the closeness estimates need to be uniform in $\bm{\alpha}$. The main difficulty is that $\omega$-dependence enters not only through $ \bm{v}_{\omega,\bm{\alpha}}$, but also through the spaces $N_\omega, R_\omega$ and their orthogonal complements. We start by showing that \eqref{proj 0} has no solution.

\begin{lemma} \label{LEM proj 0}
Assume the notation introduced above and that $\chi_2(x)>0$ for a.e. $x \in D$, then there exists a constant $C=C(\chi_2) > 0$, such that

\begin{equation} \label{proj 0 lower bound}
\big\| P_{R_0^\perp} A(\bm{v}) \big\|_X \geq C \|\bm{v} \|_X^2, \qquad \qquad \forall \ \bm{v} \in N_0.
\end{equation}
\end{lemma}

\begin{remark}
If $\chi_2$ is a positive constant, then the implicit constant in \eqref{proj 0 lower bound} is of the form $C(\chi_2) = C_0 \chi_2 $, where $C_0>0$ is an absolute constant.
\end{remark}

\begin{proof}
The inequality is obvious for $\bm{v} = 0$, so let us assume that $\bm{v} \neq 0$. Note that $N_0 = \text{span}\{(1,0); (x,0)\}$ and $R_0^\perp = N_0$, hence $\bm{v} = (\alpha_1 + \alpha_2 x, 0)$ for some $\alpha_1, \alpha_2 \in \CC$. Let $\bm{f}=(f,0)$ and $\bm{g}=(g,0)$ be an $X$-orthonormal basis for $N_0$, then the orthogonal projection operator is given by

\begin{equation*}
P_{N_0} A(\bm{v}) = \langle A(\bm{v}), \bm{f} \rangle_X \ \bm{f} + \langle A(\bm{v}), \bm{g} \rangle_X \ \bm{g},
\end{equation*}

\n but the definition of $A$ \eqref{A} implies

\begin{equation*}
\langle A (\bm{v}), \bm{f} \rangle_X = 
\int_{D} \chi_2 v_1^2  \overline{f} dx
\end{equation*}

\n and analogously for $\bm{g}$. In particular, $P_{N_0} A(\bm{v})$ depends only on $\chi_2$ and is independent of $\chi_1$. For any $\lambda \in \CC$ and $\bm{v} \in N_0$ 

\begin{equation*}
P_{N_0} A(\lambda \bm{v}) = \lambda^2 P_{N_0} A(\bm{v}). 
\end{equation*}

\n Therefore, establishing the inequality \eqref{proj 0 lower bound} is equivalent to showing 

\begin{equation} \label{proj A lower bound 2}
\|P_{N_0} A(\bm{v}) \|_X \geq C >0 \qquad \qquad \forall \ \|\bm{v}\|_X = 1.
\end{equation}

\n The map $\bm{v} \to \|P_{N_0} A(\bm{v}) \|_X$ is continuous as a map $N_0 \to \RR$. The set $S = \{\bm{v} \in N_0: \|\bm{v}\|_X = 1\}$ is closed and bounded, and therefore (by finite dimensionality) compact. In particular, $\|P_{N_0} A(\cdot) \|_X$ attains its minimum on $S$. Hence, \eqref{proj A lower bound 2} will follow once we show 

\begin{equation} \label{proj A lower bound 3}
\|P_{N_0} A(\bm{v}) \|_X \neq 0 \qquad \qquad \forall \ \|\bm{v}\|_X = 1.
\end{equation}

\n It is clear that $P_{N_0} A(\bm{v}) = 0$ iff $A(\bm{v})$ is orthogonal to $(1,0)$ and $(x,0)$, i.e.,

\begin{equation*}
\int_D \chi_2 (\alpha_1 + \alpha_2 x)^2 dx = 0, \qquad \text{and} \qquad \int_D \chi_2 (\alpha_1 + \alpha_2 x)^2 x dx = 0.
\end{equation*}  

\n Let us show that these equations imply that $\bm{v} = 0$, which then verifies (\ref{proj A lower bound 3}). If $\alpha_2 = 0$, then obviously $\alpha_1 = 0$, because $\chi_2>0$ a.e. in $D$, and we conclude that $\bm{v}=0$. So let us assume that $\alpha_2 \neq 0$, and set $\alpha = \alpha_1 / \alpha_2$, then we obtain

\begin{equation} \label{alpha +x two integrals}
\int_D \chi_2 (\alpha + x)^2 dx = 0, \qquad \text{and} \qquad \int_D \chi_2 (\alpha + x)^2 x dx = 0.
\end{equation}  

\n These are two quadratic equations in $\alpha$. They may be rewritten as

\begin{equation} \label{alpha system}
\begin{cases}
\displaystyle  \alpha^2 \int_D \chi_2 dx + 2\alpha \int_D \chi_2 x dx + \int_D \chi_2 x^2 dx = 0
\\[.2in]
\displaystyle \alpha^2 \int_D \chi_2 x dx + 2 \alpha \int_D \chi_2 x^2 dx + \int_D \chi_2 x^3 dx = 0.
\end{cases}
\end{equation}

\n If $\alpha \in \CC$ is a solution of the above system, then we shall show that $\alpha \in \RR$. Before doing so, note that this concludes our proof since  \eqref{alpha +x two integrals} obviously cannot hold for $\alpha$ real (thus leading us to the conclusion that $\alpha_2=\alpha_1=0$). 

Indeed if $\alpha \in \CC$ is a solution to \eqref{alpha system}, then we can cancel the $\alpha^2$ terms by multiplying the top equation  by $\frac12 \int_D \chi_2 x dx$ and the bottom equation by $\frac12 \int_D \chi_2 dx$ and subtracting the bottom from the top. This yields a linear equation for $\alpha$ with real coefficients and leading coefficient given by 

\begin{equation*}
c =   \left( \int_D \chi_2 x dx \right)^2 - \int_D \chi_2 x^2 dx \int_D \chi_2 dx.  
\end{equation*}

\n Since $\chi_2 > 0$ a.e. in $D$, the Cauchy-Schwartz inequality gives that $c \neq 0$ (in fact, $c<0$). It follows that $\alpha$ must be real. This completes the proof of Lemma \ref{LEM proj 0}.
\end{proof}

\begin{remark} \label{REM proj_0 diamond}
\normalfont
For the operator $T^\diamond: X_\diamond \to X_\diamond$ given by \eqref{T diamond} we similarly write

\begin{equation*}
T^\diamond = L^\diamond + \omega^2 A^\diamond,
\qquad \qquad
\langle A^\diamond (\bm{u}), \bm{\phi} \rangle_{X_\diamond} = 
\int_{D} \frac{\chi_1}{x} u_2 \overline{u}_1 \overline{\phi}_2 + \frac{\chi_2}{x} u_1^2  \overline{\phi}_1 dx.
\end{equation*}

\n In view of Remark~\ref{REM L diamond N and R^perp} $(R_0^\diamond)^\perp = N_0^\diamond = \text{span}\left\{(x,0)\right\}$. Therefore, 

\begin{equation*}
P_{N_0^\diamond} A^\diamond(\bm{v}) = 3 \langle A^\diamond(\bm{v}), (x,0) \rangle_{X_\diamond} \, (x,0).
\end{equation*}

\n But $\bm{v} \in N_0^\diamond$ is given by $\bm{v} = (\alpha x, 0)$ for some $\alpha \in \CC$, and a direct calculation gives

\begin{equation*}
\|P_{N_0^\diamond} A^\diamond(\bm{v})\|_{X_\diamond} = \sqrt{3}|\alpha|^2 \int_D \chi_2 x^2 dx = 3 \sqrt{3} \int_D \chi_2 x^2 dx \, \|\bm{v}\|_{X_\diamond}^2 .
\end{equation*}

\n Thus, Lemma~\ref{LEM proj 0} holds trivially in this case. Moreover, we see that if $\chi_2(x) \geq m > 0$ for a.e. $x\in D$, then the constant in the above inequality (or equivalently, in the analogue of \eqref{proj 0 lower bound}) can be taken to depend just on $m$.
\end{remark}

\vspace{.1in}

Let us next state some auxiliary lemmas regarding closeness estimates for the problems \eqref{proj 0} and \eqref{proj omega}. In particular we will assume that $\omega$ is small enough, specifically that it lies, in the disk

\begin{equation} \label{Omega_Q}
\Omega_Q = \left\{ \omega \in \CC: |\omega| \leq \frac{1}{2 \sqrt{Q}} \right\},
\end{equation}

\n where $Q$ is given by \eqref{q assumption}. The proofs of the lemmas below will be given in the Appendix.

\begin{lemma} \label{LEM PN_omega}
Let $N_\omega \subset X$ be given by \eqref{N and R^perp L} and $P_{N_\omega}: X \to N_\omega$ be the orthogonal projection operator. Then the map $\omega \mapsto P_{N_\omega}$ is continuous as a map from $\Omega_Q \to \CL(X)$. Moreover, there exist $\delta, C>0$ depending only on $Q$, such that

\begin{equation} \label{P_omega-P0 bound omega^2}
\left\| P_{N_\omega} - P_{N_0} \right\|_{\CL(X)} \leq C |\omega|^2, \qquad \qquad \forall \ |\omega|<\delta.
\end{equation}

\end{lemma}

\begin{remark}
The lemma also holds for $R_\omega^\perp$ in place of $N_\omega$.
\end{remark}

\begin{lemma} \label{LEM gamma_omega}
Let $L_\omega$ be given by \eqref{L} and let $N_\omega, R_\omega$ be its kernel and its range, respectively. Let also

\begin{equation*}
\gamma_\omega = \left\| L_\omega^{-1} \right\|_{\CL(R_\omega, N_\omega^\perp)},
\end{equation*}

\n then $\omega \mapsto \gamma_\omega$ is continuous as a function from $\Omega_Q \to (0, \infty)$. In particular, as $\omega \to 0$

\begin{equation*}
\gamma_\omega \longrightarrow \gamma_0 = \left\| L_0^{-1} \right\|_{\CL(R_0, N_0^\perp)}.
\end{equation*}
\end{lemma}

\begin{lemma} \label{LEM v_omega - v_0}
Let $\bm{v}_{\omega, \bm{\alpha}}$ be given by \eqref{v_omega basis expansion} with $\omega \in \Omega_Q, \ \bm{\alpha} \in \CC^2$. There exist constants  $\sigma_1 > 0$ and $\sigma_2>0$, which depend only on $Q$, such that

\begin{equation*}
\|\bm{v}_{\omega, \bm{\alpha}} - \bm{v}_{0,\bm{\alpha}}\|_X \leq |\omega|^2 \sigma_1  \|\bm{v}_{\omega, \bm{\alpha}}\|_X , 
\qquad \qquad \forall \ \bm{\alpha} \in \CC^2,
\end{equation*}

\n and

\begin{equation*}
\|\bm{v}_{0,\bm{\alpha}}\|_X^2 \geq \sigma_2  \|\bm{v}_{\omega, \bm{\alpha}}\|_X^2 , 
\qquad \qquad \forall \ \bm{\alpha} \in \CC^2.
\end{equation*}
\end{lemma}

\vspace{.1in}

Using the above lemmas we now establish a key result that will be used in the proof of the main theorem.

\begin{lemma} \label{LEM proj omega}
Suppose that \eqref{M and chi2} and \eqref{q assumption} hold, and let $M$ and $Q$ be defined as in \eqref{M and chi2} and \eqref{q assumption}. Let $r>0$ and $\omega \in \Omega_Q$ be such that \eqref{contraction cond} holds. Let $\bm{v}_\omega \in N_\omega$ be given by \eqref{v_omega basis expansion} (in the notation we suppress its dependence on $\bm{\alpha}$) and assume $\|\bm{v}_\omega\|_X \leq r$. Finally, let $\bm{w}=\bm{w}(\bm{v}_\omega)\in \overline{B_r^{N^\perp_\omega}}(0)$ be defined by the fixed point equation \eqref{w fixed point eq and B}. 
There exists a positive constant $c_Q(\chi_2)$ (depending only on $Q$ and $\chi_2$) such that if in addition

\begin{equation} \label{omega^2 bound lemma}
|\omega|^2 \leq \frac{c_Q(\chi_2)}{M (1 + M r)},
\end{equation}

\n then

\begin{equation*}
\|P_{R_\omega^\perp} A (\bm{v}_\omega + \bm{w}(\bm{v}_\omega))\|_X \geq c_Q(\chi_2) \|\bm{v}_\omega\|^2_X.
\end{equation*}

\end{lemma}

\begin{remark}
If the nonlinear susceptibility $\chi_2$ is a positive constant, then in the above lemma, one can take $c_Q(\chi_2) = c_q \chi_2 $, where the constant $c_q>0$ only depends on $Q$.
\end{remark}

\begin{proof}
Note that $\bm{v}_0$ is given by the formula \eqref{v_omega basis expansion} with $\omega = 0$, yielding $\bm{v}_0=\alpha_1+\alpha_2 x$. To somewhat simplify the notation let us write $\bm{w} = \bm{w}(\bm{v}_\omega)$. We start with the inequality

\begin{equation*}
\begin{split}
\|P_{R_\omega^\perp} A (\bm{v}_\omega + \bm{w}) - P_{R_0^\perp} A (\bm{v}_0) \|_X &\leq \left\|P_{R_\omega^\perp} \left[ A(\bm{v}_\omega + \bm{w}) - A(\bm{v}_0) \right] \right\|_X + \left\| \left( P_{R_\omega^\perp} - P_{R_0^\perp} \right) A(\bm{v}_0) \right\|_X 
\\
&=: I_1 + I_2.
\end{split}
\end{equation*}

\n Invoking the first inequality of Lemma~\ref{LEM A Lip}, we can estimate

\begin{equation*}
I_2 \leq C M p(\omega) \|\bm{v}_0\|_X^2,
\qquad \qquad
p(\omega) = \|P_{R_\omega^\perp} - P_{R_0^\perp}\|_{\CL(X)}.
\end{equation*}

\n Using the second inequality of the same lemma we also have

\begin{equation} \label{I1}
I_1 \leq C M \max\{\|\bm{v}_0\|_X, \|\bm{v}_\omega\|_X, \|\bm{w}\|_X\} \big( \|\bm{v}_\omega - \bm{v}_0\|_X +  \|\bm{w}\|_X \big).
\end{equation}

\n Lemma~\ref{LEM v_omega - v_0} implies the bound 

\begin{equation} \label{v_omega -v0 in proof} 
\|\bm{v}_\omega - \bm{v}_0\|_X \leq |\omega|^2 \sigma_1 \|\bm{v}_\omega \|_X,
\end{equation}

\n which in turn gives

\begin{equation} \label{v0 by vomega}
\|\bm{v}_0\|_X \leq \left(1 + |\omega|^2  \sigma_1\right)  \|\bm{v}_\omega \|_X.
\end{equation}

\n The last estimate will be used inside the maximum in \eqref{I1}. Next, we can estimate $\bm{w}$ by using the inequality \eqref{w(v) v^2 bound} of Lemma~\ref{LEM B}, which reads

\begin{equation} \label{w v_omega^2}
\|\bm{w}\|_X \leq 4C M |\omega|^2 \gamma_\omega \|\bm{v}_\omega\|_X^2,
\end{equation}

\n where $\gamma_\omega$ is defined by \eqref{gamma}. The hypothesis \eqref{contraction cond}, combined with the above inequality also gives

\begin{equation} \label{w 1/4 v}
\|\bm{w}\|_X \leq  \|\bm{v}_\omega\|_X.
\end{equation}
\n We may without loss of generality assume that all the absolute constants, $C$, in the above inequalities are the same. Let us now use \eqref{v0 by vomega} and \eqref{w 1/4 v} to bound the maximum in \eqref{I1} and let us use \eqref{v_omega -v0 in proof} and \eqref{w v_omega^2} to bound the remaining terms in \eqref{I1}. We arrive at 

\begin{equation*}
I_1 \leq C M |\omega|^2 \left(1 + |\omega|^2  \sigma_1 \right) \big[ \sigma_1 + 4 C M \gamma_\omega r \big] \ \|\bm{v}_\omega\|_X^2.
\end{equation*}

\n Using \eqref{v0 by vomega} in the bound for $I_2$ we get

\begin{equation*}
I_2 \leq C M p(\omega) \big[ 1 +|\omega|^2 \sigma_1 \big]^2 \ \|\bm{v}_\omega\|_X^2.
\end{equation*}

\n Invoking Lemma~\ref{LEM proj 0}, for a constant $C_2:=C(\chi_2)>0$ we have 

\begin{equation} \label{z0 lower bound}
\|P_{R_0^\perp} A (\bm{v}_0)\|_X \geq C_2 \|\bm{v}_0\|_X^2 \geq C_2 \sigma_2 \|\bm{v}_\omega\|_X^2, 
\end{equation}

\n where in the last step we used Lemma~\ref{LEM v_omega - v_0}. Now, using the triangle inequality and combining the above inequalities we obtain

\begin{equation*}
\|P_{R_\omega^\perp} A (\bm{v}_\omega + \bm{w})\|_X \geq \|P_{R_0^\perp} A (\bm{v}_0) \|_X - I_1 - I_2 \geq \kappa \|\bm{v}_\omega\|_X^2,
\end{equation*}  

\n where

\begin{equation*}
\kappa = C_2 \sigma_2 - C M \left( 1 +|\omega|^2 \sigma_1 \right) \left[ |\omega|^2 \left( \sigma_1 + 4C M \gamma_\omega r \right) + p(\omega) \left( 1 +|\omega|^2 \sigma_1 \right)   \right].
\end{equation*}

\n It remains to make sure that $\kappa$ is positive, e.g. by requiring

\begin{equation*}
\kappa \geq \frac{1}{2} C_2 \sigma_2,  
\end{equation*}

\n which is insured provided

\begin{equation} \label{long}
|\omega|^2 \left[ \sigma_1 + 4C M \gamma_\omega r \right] + p(\omega) \left[ 1 +|\omega|^2 \sigma_1 \right]
\leq
\frac{C_2 \sigma_2}{2C M \left( 1 +|\omega|^2 \sigma_1 \right)}.
\end{equation}

\n Lemma~\ref{LEM PN_omega} asserts that $p(\omega) \leq C_3 |\omega|^2$ for $|\omega|<\delta$, with $C_3$ and $\delta>0$ depending only on $Q$. Therefore, by continuity,  the function $p(\omega)/ |\omega|^2$ is bounded on the disk $\Omega_Q$. Bounding $\gamma_\omega$ (which is also continuous) by its maximum on $\Omega_Q$ we now have that, for some constant $C_Q$

\begin{equation*}
|\omega|^2 \left[ \sigma_1 + 4 C M \gamma_\omega r \right] + p(\omega) \left[ 1 +|\omega|^2 \sigma_1 \right]
\leq C_Q |\omega|^2 (1+M r).
\end{equation*}

\n Similarly, the right hand side of \eqref{long} can be bounded from below by its minimum value on $\Omega_Q$:

\begin{equation*}
\frac{C_2 \sigma_2}{2C M \left( 1 +|\omega|^2 \sigma_1 \right)} \geq \frac{\tilde{c}_Q C_2}{M}.
\end{equation*}

\n Thus, \eqref{long} will be satisfied, if we impose

\begin{equation*}
C_Q |\omega|^2 (1+M r) \leq \frac{\tilde{c}_Q C_2}{M}, ~~\hbox{ or } ~ |\omega|^2\leq \frac{\tilde{c}_Q C_2}{C_Q M(1+Mr)}.
\end{equation*}

\n The assertion of this lemma now follows if we take $c_Q(\chi_2)=\min \{ \frac{\tilde c_Q C_2}{C_Q},\frac12 C_2\sigma_2  \}$.

\end{proof}

\subsection{Concluding the proof of Theorem~\ref{THM 1}}

Fix $0 \neq \bm{u}$ and let $r = \|\bm{u}\|_X$. Suppose that $\omega \in \Omega_Q$, where $\Omega_Q$ is the disk in the complex plane given by \eqref{Omega_Q}. Further, suppose that \eqref{contraction cond} holds, more precisely suppose

\begin{equation*}
|\omega|^2 < \frac{1}{\displaystyle 4CM r \max_{\omega \in \Omega_Q} \gamma_\omega} =: \frac{C_Q}{M r},
\end{equation*}

\n Recall that $\gamma_\omega$ is a continuous function of $\omega \in \Omega_Q$ (cf. Lemma~\ref{LEM gamma_omega}) so that the above maximum is attained. Finally, assume the inequality \eqref{omega^2 bound lemma} also holds. In other words, putting these three conditions together we are assuming that

\begin{equation} \label{omega^2 min}
|\omega|^2 \leq \min \left\{ \frac{C_Q}{M r}, \frac{c_Q(\chi_2)}{M (1+Mr)}, \frac{1}{4Q} \right\}.
\end{equation}

\n For the sake of a contradiction, let us now assume $T_\omega(\bm{u}) = 0$. Set $\bm{v} = P_{N_\omega} \bm{u}$, then Theorem~\ref{THM projection} implies that we must have $\bm{u} = \bm{v} + \bm{w}(\bm{v})$ and 

$$P_{R_\omega^\perp} A (\bm{v} + \bm{w}(\bm{v})) = 0.$$ 

\n But Lemma~\ref{LEM proj omega} gives the estimate

\begin{equation*}
\|P_{R_\omega^\perp} A (\bm{v} + \bm{w}(\bm{v}))\|_X \geq c_Q(\chi_2) \|\bm{v}\|^2_X.
\end{equation*}

\n This implies that $\bm{v}=0$, and consequently $\bm{w}(\bm{v})=0$ and $\bm{u}=0$, which is a contradiction. Note that if we define $\tilde c_Q(\chi_2)= \min \{ C_Q  \|\chi_2\|_{L^\infty}, c_Q(\chi_2), \| \chi_2\|_{L^\infty}/4 Q\}$ then the condition \eqref{omega^2 min} is satisfied, provided

\begin{equation*}
|\omega|^2 \leq \frac{\tilde c_Q(\chi_2)}{M (1+Mr)}.
\end{equation*}

\n This completes the proof of Theorem \ref{THM 1}.

\section{Existence in the presence of a force term}\label{section5}
\setcounter{equation}{0} 

\subsection{The main idea and the outline of the proof}

The goal here is to prove Theorem~\ref{THM exist}. Let us separate out the first nonlinearity in \eqref{trans rad f}, which contains complex conjugation:

\begin{equation*}
T^\diamond = F + G,
\end{equation*}

\n where $F, G$ are given by the following formulas (as usual, sometimes we will write $F_\omega$ to indicate the dependence of $F$ on $\omega$)

\begin{equation} \label{F}
\langle F (\bm{u}), \bm{\phi} \rangle_{X_{\diamond}} = \int_{D} (u_1'' + \omega^2 q u_1)  \overline{\phi}_2 + u_2 \left( \overline{\phi}_1'' + 4 \omega^2 q \overline{\phi}_1 \right) dx + \omega^2 \int_{D} \chi_2 u_1^2 \overline{\phi}_1 \, \frac{dx}{x}
\end{equation}

\n and

\begin{equation} \label{G}
\langle G (\bm{u}), \bm{\phi} \rangle_{X_{\diamond}} = \omega^2 \int_{D} \chi_1 u_2 \overline{u}_1 \overline{\phi}_2 \, \frac{dx}{x}.
\end{equation}

\n The reason for separating out the nonlinearity containing conjugation is that now $F:X_\diamond \to X_\diamond$ is Fr{\'e}chet differentiable, in fact $C^1$, and for any $\bm{p}, \bm{u} \in X_\diamond$

\begin{equation} \label{F linearization}
F(\bm{p} + \bm{u}) = F(\bm{p}) + F'_{\bm{p}}(\bm{u}) + E(\bm{u}),
\end{equation}

\n where $F'_{\bm{p}}: X_\diamond \to X_\diamond$ denotes the Fr{\'e}chet derivative of $F$ at the point $\bm{p}$ and is given by

\begin{equation} \label{F'_p}
\langle F'_{\bm{p}} (\bm{u}), \bm{\phi} \rangle_{X_{\diamond}} = \int_{D}\left[ (u_1'' + \omega^2 q u_1)  \overline{\phi}_2 + u_2 \left( \overline{\phi}_1'' + 4 \omega^2 q \overline{\phi}_1 \right)\right] dx + 2\omega^2 \int_{D} \chi_2 p_1 u_1 \overline{\phi}_1 \, \frac{dx}{x}.
\end{equation}

\n Often, we will also write $F'_{\omega, \bm{p}}$ to indicate the dependence on $\omega$. Finally, $E : X_\diamond \to X_\diamond$ is the error term given by

\begin{equation} \label{E}
\langle E(\bm{u}), \bm{\phi} \rangle_{X_\diamond} = \omega^2 \int_{D} \chi_2 u_1^2 \overline{\phi}_1  \, \frac{dx}{x}.
\end{equation}

We are going to search for solutions to the problem \eqref{trans rad f} near the point $\bm{p}$ (which for now is an arbitrary nonzero point), i.e., we consider the equation $T^\diamond(\bm{p} + \bm{u}) = \bm{f}$, where $\bm{u}$ is the unknown variable, assumed to be small. The reason for introducing the point $\bm{p}$ is to create an invertible operator so that the original equation can be rewritten as a fixed point equation. Let us rewrite our equation as

\begin{equation*}
F(\bm{p}+\bm{u}) + G(\bm{p}+\bm{u}) = \bm{f}.
\end{equation*}

\n In Lemma~\ref{LEM F' invertible} below, it will be shown that under a suitable condition on $\omega$ the operator $F'_{\bm{p}}$ is invertible. But then, using \eqref{F linearization} in the above equation and inverting $F'_{\bm{p}}$, we obtain a fixed point equation

\begin{equation} \label{u=H(u)}
\bm{u} = H(\bm{u}),
\end{equation}

\n where

\begin{equation*}
H(\bm{u}) = (F'_{\bm{p}})^{-1} \left[\bm{f} - F(\bm{p}) - E(\bm{u}) - G(\bm{p} + \bm{u}) \right].
\end{equation*}

\n We will utilize the freedom to choose $\bm{p}$, in particular it will be chosen so that $F(\bm{p}) = \bm{f}$. This equation is easier to analyze, as it corresponds to two decoupled ODEs. Our approach is thus as follows: (1) for a given $0\neq f \in L^2$ and $\bm{f} = (0,f)$ find $\bm{p}$ such that

\begin{equation} \label{F(p)=f and F' invert}
F_\omega(\bm{p}) = \bm{f}, \qquad \text{and} \qquad F'_{\omega, \bm{p}} \quad \text{is invertible}.
\end{equation}

\n (2) solve the equation $\bm{u}=H(\bm{u})$ for this choice of $\bm{p}$. To achieve (1) we require a certain restriction on $\omega$, namely that $\omega \notin \Lambda$, where $\Lambda$ is given by (\ref{Lambda}).

\begin{remark} \label{REM choice f=0}
\normalfont
Such a choice of $\bm{p}$ is not possible when $f=0$. The reason is the scaling symmetry of the equation $F(\bm{p}) = 0$, which implies that if $\bm{p}$ is a solution of this equation, then there are other solutions accumulating near $\bm{p}$. This immediately prevents $F'_{\bm{p}}$ from being invertible.
\end{remark}

\n Note that Theorem~\ref{THM exist} provides an existence result for any $\omega \in \CG$, with a smallness assumption on $\| \chi_1 \|_{L^\infty}$ which doesn't depend on $\omega$. Therefore when analyzing the fixed point equation \eqref{u=H(u)} we need to ensure that $H$ is a contraction for all $\omega \in \CG$. The main difficulty stems from the dependence of $\bm{p}$ on $\omega$. Indeed, the operator $F$ depends on $\omega$, hence so does the solution $\bm{p}$ of the equation $F(\bm{p}) = \bm{f}$. Consequently, the operator $(F'_{\omega, \bm{p}})^{-1}$ depends on $\omega$ also through $\bm{p}$, and we need to control its operator norm, uniformly in $\omega$. This is where the ``buffer zone" excluded from the set $\CG$ \eqref{G set} will be essential.

\subsection{The equation $F_\omega(p) = f$ and invertibility of $F'_{\omega, p}$}

Let $z_\omega$ be defined by \eqref{z_omega}. Recall that $z_\omega$ is a basis for the kernel of $L^\diamond$ and $z_{2{\overline{\omega}}}$ is a basis for the orthogonal complement of the range of $L^\diamond$, where $L^\diamond$ is the linearization of $T^\diamond$ at $0$, cf. Remark~\ref{REM L diamond N and R^perp}.

\begin{lemma} \label{LEM z_omega properties}
Let $z_\omega(x)$ be defined by \eqref{z_omega}, then $\omega \mapsto z_\omega(x)$ is an entire function of $\omega$ for each fixed $x \in [0,1]$. Moreover, if $D \subset \CC$ is a compact set, then the functions $z_\omega(x), z_\omega'(x)$ and $\partial_\omega z_\omega(x)$ are bounded, as functions of two variables $(\omega, x)$ in the set $D \times [0,1]$. 
\end{lemma}

This is a classical result about solutions of differential equations involving a complex parameter. For the proof we refer to \cite{olver74}.This lemma will be used in the sequel. We now address the invertibility of the operator $F'_{\omega, \bm{p}}$ defined by \eqref{F'_p}.

\begin{lemma} \label{LEM F' invertible}
Let $\bm{p} = (p_1, p_2) \in X_\diamond$ and $0 \neq \omega \in \CC$ be such that 

\begin{equation*}
\int_0^1 p_1 z_\omega z_{2\omega}  \, \frac{\chi_2 dx}{x} \neq 0,
\end{equation*}

\n then $F'_{\omega, \bm{p}}$ is an invertible operator on $X_\diamond$.
\end{lemma}

\begin{proof}
Note that $F'_{\omega, \bm{p}}$ is the sum of the operator $L^\diamond$ (defined by \eqref{L}) and a compact operator. We showed in Theorem~\ref{THM L} and Remark~\ref{REM L diamond N and R^perp} that $L^\diamond:X_\diamond \to X_\diamond$ is a Fredholm operator of index 0. Consequently, $F'_{\omega, \bm{p}}$ is also a Fredholm operator of index 0 on the space $X_\diamond$. Therefore, to prove the invertibility it is enough to show that $F'_{\omega, \bm{p}}$ is injective. To that end suppose $F'_{\omega, \bm{p}}(\bm{u}) = 0$ and let us show that $\bm{u}=0$. Directly form the definition we have

\begin{equation} \label{F'=0 equations}
\begin{split}
\int_0^1 (u_1'' + \omega^2 q u_1)  \overline{\phi}_2 dx &= 0, \qquad \qquad \forall \phi_2 \in L^2
\\[.05in]
\int_0^1 u_2 \left[ \left( \overline{\phi}_1'' + 4 \omega^2 q \overline{\phi}_1 \right) +  2 \omega^2 \frac{\chi_2}{x} p_1 u_1 \overline{\phi}_1 \right] dx &= 0, \hspace{0.65in} \forall \phi_1 \in H^2_\diamond.
\end{split}
\end{equation}

\n The first equation clearly implies that $u_1'' + \omega^2 q u_1 = 0$ and since $u_1 \in H^2_\diamond$, we conclude that $u_1 = \alpha z_\omega$ for some $\alpha \in \CC$. Suppose $\alpha\ne 0$, and take $\phi_1 = z_{-2 \overline{\omega}}$. Since $\phi_1'' + 4 \overline{\omega}^2 q \phi_1 = 0$ the second equation of \eqref{F'=0 equations} implies

\begin{equation*}
\int_0^1 \frac{\chi_2}{x} p_1 z_\omega \, \overline{z_{-2{\overline{\omega}}}} \, dx = 0,
\end{equation*}

\n which contradicts the hypothesis of the lemma, because $\overline{z_{-2{\overline{\omega}}}} = z_{2\omega}$  (cf. \eqref{z sym}). We conclude that $\alpha$ must be zero and thus $u_1=0$. It now follows immediately from the second equation in (\ref{F'=0 equations}) that $u_2=0$, and as a consequence $\bm{u}=0$.

\end{proof}

Now we turn to solving the equation $F_\omega(\bm{p}) = \bm{f}$. Let $a_\omega$ be as in \eqref{a_omega} and introduce the functions

\begin{equation*}
k_1(\omega) = \int_0^1 z_\omega^2 z_{2\omega} \, \frac{\chi_2 dx}{x},
\qquad
k_2(\omega) = \int_0^1 z_\omega  a_\omega z_{2\omega} \, \frac{\chi_2 dx}{x},
\qquad
k_3(\omega) = \int_0^1 a_\omega^2 z_{2\omega} \, \frac{\chi_2 dx}{x}
\end{equation*}

\n Note that the function $\zeta$ defined by \eqref{zeta} can be rewritten as

\begin{equation} \label{zeta in terms of k_j}
\zeta(\omega) = k_2^2(\omega) - k_1(\omega) k_3(\omega).
\end{equation}

\n Using Lemma~\ref{LEM z_omega properties} and the estimate \eqref{u/x bounded by H^2 norm} we conclude that $k_j(\omega)$ are entire functions for $j=1,2,3$.

\begin{lemma} \label{LEM F(p)=f}
Let $f \in L^2$ and set $\bm{f} = (0, f)$, let $\CG = \CG_{\rho, \delta} \subset \CC$ be as in \eqref{G set}. Then we can decompose 

\begin{equation} \label{G set union G_j}
\CG = \bigcup_{j=1}^s \CG_j,
\end{equation}

\n where $s$ is some integer and $\CG_j$ are regular\footnote{i.e. $\CG_j$ is the closure of its interior} compact sets, such that for each $j=1,..,s$ and $\omega \in \CG_j$ there exists $\bm{p}_{\omega} \in X_\diamond$  with the following properties:

\begin{enumerate}
\item[(i)] $F(\bm{p}_{\omega}) = \bm{f}$ and $F'_{\omega, \bm{p}_{\omega}}$ is invertible. 

\item[(ii)] $\omega \mapsto \bm{p}_{\omega}$ is continuous as a mapping $\CG_j \to X_\diamond$.

\end{enumerate}

\n Furthermore, there exists $C = C(\rho, \delta, f, q, \chi_2)>0$ such that for any $j=1,..,s$ and any $\omega \in \CG_j$

\begin{equation} \label{F' inverse unif bound}
\|(F'_{\omega, \bm{p}_{\omega}})^{-1}\|_{\CL(X_\diamond)} \leq C.
\end{equation}

\end{lemma}

\begin{proof}
The equation $F_\omega(\bm{p}) = \bm{f}$ is the weak formulation of the problem

\begin{equation} \label{F(p)=f odes}
\begin{cases}
\displaystyle p_1'' + \omega^2 q p_1  = f, \hspace{0.7in} &\text{in} \ (0,1)
\\[.1in]
\displaystyle p_2'' + 4 \omega^2 q p_2  = - \frac{\omega^2}{x} \chi_2  p_1^2, &\text{in} \ (0,1)
\\[.1in]
p_1(0) = p_2(0) = p_2(1) = p_2'(1) = 0. 
\end{cases}
\end{equation}

\n Note that the two ODEs in the above system are decoupled, therefore we can easily analyze this system. We start by solving the first equation along with its corresponding boundary condition:

\begin{equation*}
p_1 = c z_\omega + a_\omega, \qquad \qquad c \in \CC.
\end{equation*}

\n To analyze the second equation, let $b_\omega$ be given by 

\begin{equation} \label{b_omega}
b_\omega(x) = \omega^2 \int_0^x \big[ z_{2\omega}(x) \widetilde{z}_{2\omega}(t) - z_{2\omega}(t) \widetilde{z}_{2\omega}(x) \big] \chi_2(t) \big( c z_\omega(t) + a_\omega(t) \big)^2 \, \frac{dt}{t}
\end{equation}

\n and for some constant $\widetilde{c} \in \CC$ let

\begin{equation*}
p_2 = \widetilde{c} z_{2\omega} - b_\omega.
\end{equation*}

\n Then $p_2$ satisfies the second equation in \eqref{F(p)=f odes} and the condition $p_2(0) = 0$. We therefore need to choose $c$ and $\widetilde{c}$ such that the remaining boundary conditions $p_2(1) = p_2'(1) = 0$ are satisfied, i.e., 

\begin{equation*}
\begin{cases}
\widetilde{c} \, z_{2\omega}(1) = b_\omega(1)
\\
\widetilde{c} \, z_{2\omega}'(1) = b_\omega'(1) ~.
\end{cases}
\end{equation*}

\n A necessary and sufficient condition for this to have a solution, $\widetilde{c}$,  is that 
\begin{equation*}
z_{2\omega}(1) b_\omega'(1) = z_{2\omega}'(1) b_\omega(1).    
\end{equation*}

\n Using basic algebra and the fact that the Wronskian of $z_{2\omega}, \widetilde{z}_{2\omega}$ is equal to $-1$, the last equation reduces to the following equation for $c$

\begin{equation} \label{c eq1}
\int_0^1 z_{2\omega} \big( c z_\omega + a_\omega \big)^2 \chi_2 \, \frac{dx}{x} = 0.
\end{equation}

\n If we want to invoke Lemma~\ref{LEM F' invertible} to insure the invertibility of  $F'_{\omega, \bm{p}}$, we must also require

\begin{equation} \label{c eq2}
\int_0^1 z_\omega z_{2\omega} \big( c z_\omega + a_\omega \big) \chi_2 \, \frac{dx}{x} \neq 0.
\end{equation}

\n The equations \eqref{c eq1} and \eqref{c eq2} can be rewritten as

\begin{equation} \label{c system}
\begin{cases}
k_1(\omega) c^2 + 2k_2(\omega) c + k_3(\omega) = 0
\\
k_1(\omega) c + k_2(\omega) \neq 0
\end{cases}
\end{equation}

\n It is easy to see that the system \eqref{c system} is solvable for $c \in \CC$ iff the discriminant of the quadratic is nonzero, i.e., $\zeta(\omega) \neq 0$. Let now

\begin{equation*}
\bm{p}_\omega = \big( c_\omega z_\omega + a_\omega, \, \widetilde{c}_\omega z_{2\omega} - b_\omega \big),
\end{equation*}

\n where the complex constants $c_\omega, \widetilde{c}_\omega$ are defined by

\begin{equation} \label{c and tilde c}
c_\omega = 
\begin{cases}
\displaystyle -\frac{k_3(\omega)}{2k_2(\omega)}, \qquad &\text{if} \ k_1(\omega) = 0
\\[.2in]
\displaystyle \frac{-k_2(\omega) + \sqrt{\zeta(\omega)}}{k_1(\omega)}, &\text{if} \ k_1(\omega) \neq 0
\end{cases}
\qquad \quad
\widetilde{c}_\omega =
\begin{cases}
\displaystyle \frac{b_\omega'(1)}{z_{2\omega}'(1)}, \qquad &\text{if} \ z_{2\omega}(1) = 0
\\[.2in]
\displaystyle \frac{b_\omega(1)}{z_{2\omega}(1)}, &\text{if} \ z_{2\omega}(1) \neq 0
\end{cases}
\end{equation}

\n Thus, for any $\omega \neq 0$ with $\zeta(\omega) \neq 0$ (which clearly is true for any $\omega \in \CG$) we have that $F(\bm{p}_\omega) = \bm{f}$ and $F'_{\omega, \bm{p}_{\omega}}$ is invertible. This completes the proof of part $(i)$. 

To prove part $(ii)$, we note that the coefficients $c_\omega, \widetilde{c}_\omega$ and hence also $\bm{p}_\omega$ may not be continuous for $\omega \in \CG$. To ensure continuity (in fact, holomorphicity) we will need to decompose $\CG$ into parts and on each part use a consistent and continuous ``branch" of $\bm{p}_\omega$. Since $\CG$ is bounded it contains at most finitely many (possibly no) zeros of the entire function $k_1(\omega)$ (we may also assume all of these zeros are in the interior of $\CG$, otherwise we may slightly enlarge $\CG$). Choose $\epsilon_j>0$ small enough that the disk $|\omega - \omega_j| \leq \epsilon_j$ is contained in $\CG$ and contains no other zeros of $k_1(\omega)$. Let $\omega_{1}^*, ..., \omega_{m}^*$ be the zeros (if any) in $\CG$ of the entire function $z_{2\omega}(1)$, again we may assume these zeros are all in the interior of $\CG$. For the sake of the argument, suppose these zeros do not coincide with any of $\omega_j$ (if there are coincidences a similar argument will apply). Since $z_{2\omega}'(1)$ cannot be zero at any $\omega_j^*$, we can choose $\epsilon_j^*>0$ small enough such that $z_{2\omega}'(1)$ is never zero on the disks $|\omega - \omega_j^*| \leq \epsilon_j^*$ and these disks are contained in $\CG$. Moreover, we may assume all these disks (centered at $\omega_j$ and $\omega_j^*$) are mutually disjoint. As $\CG$ is just a closed disk with finitely many open disks removed (cf. \eqref{G set}), it is clear that introducing cuts we can represent

\begin{equation} \label{G set long}
\CG = \bigcup_{j=1}^n \overline{B}_{\epsilon_j}(\omega_j) \, \cup \, \bigcup_{j=1}^m \overline{B}_{\epsilon_j^*}(\omega_j^*) \, \cup \, \bigcup_{j=1}^l \CD_j,
\end{equation}

\n where $\CD_j$ are closures of bounded, open and simply connected sets; moreover, on each of these sets $z_{2\omega}(1)$, $k_1(\omega)$ and $\zeta(\omega)$ are zero free. Consider such a set $\CD_j$. Using simply connectedness (and enlarging $\CD_j$ slightly to work with an open set) we can define a holomorphic square root for $\zeta(\omega)$, as this is a nonvanishing analytic function there. Thus, using the second formulas in \eqref{c and tilde c} we obtain holomorphic coefficients $c_\omega, \widetilde{c}_\omega$ for $\omega \in \CD_j$. Continuity of $\bm{p}_\omega$ then follows directly from Lemma~\ref{LEM z_omega properties}. On a set $\overline{B}_{\epsilon_j^*}(\omega_j^*)$ the functions $z_{2\omega}'(1)$ and $k_1(\omega)$ are zero free and the argument is analogous. Finally, on $\overline{B}_{\epsilon_j}(\omega_j)$ the function $z_{2\omega}(1)$ is zero free, hence to define $\widetilde{c}_\omega$ we use the second formula in \eqref{c and tilde c}; however, $k_1(\omega)$ has a unique zero on this set, at $\omega_j$. To make sure that $c_\omega$ defined by \eqref{c and tilde c} is continuous we choose the holomorphic branch of $\sqrt{\zeta(\omega)}$ such that $\sqrt{\zeta(\omega_j)} = k_2(\omega_j)$ so that appropriate cancellation happens, making sure that $c_\omega$ is continuous at $\omega_j$ (recall that $\zeta = k_2^2 - k_1 k_3$). It remains to denote the sets in the decomposition \eqref{G set long} by $\CG_j$.

Let us now turn to the proof of (\ref{F' inverse unif bound}).  For shorthand let us set $\Gamma_\omega = F'_{\omega, \bm{p}_\omega}$. Assume the estimate \eqref{F' inverse unif bound} is not true on $\CG_j$, then there exists a sequence of complex numbers $\{\omega_n\} \subset \CG_j$ such that, as $n \to \infty$ 

\begin{equation} \label{Gamma_j^-1 to infty}
\|\Gamma_{\omega_n}^{-1}\|_{\CL(X_\diamond)} \to \infty.
\end{equation}

\n As $\CG_j \subset \CC$ is compact, from $\{\omega_n\}$ we may extract a convergent subsequence, which we do not relabel: $\omega_n \to \omega_*$, for some $\omega_* \in \CG_j$. The definition of $\CG_j$ implies that $\omega_* \neq 0$ and $\zeta(\omega_*) \neq 0$. In particular, $\Gamma_{\omega_*}$ is an invertible operator on $X_\diamond$. Invoking part $(ii)$, we know that $\bm{p}_\omega$ is continuous for $\omega \in \CG_j$. But then it is easy to see that $\Gamma_{\omega_n} \to \Gamma_{\omega_*}$ in $\CL(X_\diamond)$. Using the invertibility of the limit $\Gamma_{\omega_*}$, it is a straightforward exercise to show that as $n\to \infty$

\begin{equation*}
\Gamma_{\omega_n}^{-1} \to \Gamma_{\omega_*}^{-1}.
\end{equation*}

\n This clearly contradicts \eqref{Gamma_j^-1 to infty}. 

\end{proof}

\subsection{Concluding the proof of Theorem \ref{THM exist} by the Banach fixed-point theorem}

Let $f \in L^2(0,1)$ with $\bm{f}=(0,f)$ and $\CG = \bigcup_{j=1}^s \CG_j$ as in \eqref{G set union G_j}. Fix $j$ and let $\omega \in \CG_j$ and $\bm{p}_\omega$ be as in Lemma~\ref{LEM F(p)=f}. Then from \eqref{u=H(u)} we see that the equation $T^\diamond(\bm{p}_\omega + \bm{u}) = \bm{f}$ can be rewritten as a fixed point equation $\bm{u} = H(\bm{u})$ where

\begin{equation} \label{H}
H(\bm{u}) = - (F'_{\omega, \bm{p}_\omega})^{-1} \left[ E(\bm{u}) + G(\bm{p}_\omega + \bm{u}) \right].
\end{equation}

\n We mention that $f$ appears implicitly in the above formulation, since $\bm{p}_\omega$ depends on $f$. Consider the difference

\begin{equation} \label{H(u)-H(v) eqn}
H(\bm{u}) - H(\bm{v}) = - (F'_{\omega, \bm{p}_\omega})^{-1} \left[ E(\bm{u}) - E(\bm{v}) + G(\bm{p}_\omega + \bm{u}) - G(\bm{p}_\omega + \bm{v}) \right].
\end{equation}

\n Assume $R > 0$ and $\bm{u}, \bm{v} \in \overline{B_R}(0) \subset X_\diamond$. Using the definitions of $E,G$ from \eqref{E} and \eqref{G} respectively, it is easy to see that for some absolute constant $C_0>0$ the following Lipschitz estimates hold:

\begin{equation} \label{2 Lip estimates}
\begin{split}
\|E(\bm{u}) - E(\bm{v})\|_{X_\diamond} &\leq C_0 R |\omega|^2 \|\chi_2\|_{L^\infty}  \|\bm{u} - \bm{v}\|_{X_\diamond} ,
\\[.07in]
\|G(\bm{p}_\omega + \bm{u}) - G(\bm{p}_\omega + \bm{v})\|_{X_\diamond} &\leq C_0 |\omega|^2 \|\chi_1\|_{L^\infty} \left(R + \|\bm{p}_\omega\|_{X_\diamond} \right) \|\bm{u} - \bm{v}\|_{X_\diamond}.    
\end{split}
\end{equation}

\n Item $(ii)$ of Lemma~\ref{LEM F(p)=f} implies that $\omega \mapsto \bm{p}_\omega$ is continuous as a map from $\CG_j \to X_\diamond$. In particular, $\|\bm{p}_\omega\|_{X_\diamond}$ reaches its maximum on $\CG_j$ and this maximum depends on $f, \chi_2, q$. Further, using \eqref{F' inverse unif bound} we can bound the operator norm of $(F'_{\omega, \bm{p}_\omega})^{-1}$ uniformly for $\omega \in \CG_j$. Thus, there exists $C = C(\CG, f, q, \chi_2)>0$ such that

\begin{equation} \label{H Lip}
\|H(\bm{u}) - H(\bm{v})\|_{X_\diamond} \leq C \left[ R \big(1+\|\chi_1\|_{L^\infty} \big) + \|\chi_1\|_{L^\infty} \right] \|\bm{u} - \bm{v}\|_{X_\diamond},
\end{equation}

\n for all $\bm{u}, \bm{v} \in \overline{B_R}(0)$ and $\omega \in \CG_j$. Further, we can also estimate

\begin{equation*}
\|H(0)\|_{X_\diamond} \leq \|(F'_{\omega, \bm{p}_\omega})^{-1}\|_{\CL(X_\diamond)} \, C_0 |\omega|^2 \|\chi_1\|_{L^\infty} \|\bm{p}_\omega\|_{X_\diamond}^2
\leq  C \|\chi_1\|_{L^\infty},  
\qquad \qquad \forall \ \omega \in \CG_j,
\end{equation*}

\n where we may assume $C$ is the same constant as the one in \eqref{H Lip}. Let us start by assuming

\begin{equation} \label{R equation}
R = \frac{1}{4C + 1}, \qquad \qquad  
\|\chi_1\|_{L^\infty} \leq \min \left\{ \frac{1}{4C},  \frac{1}{2C(4C+1)}\right\}.
\end{equation}

\n Note that we can impose the above assumption on $\chi_1$, as $C$ does not depend on $\chi_1$. This implies that the Lipschitz constant of $H$ on $\overline{B_R}(0)$ is at most $1/2$ and also

\begin{equation*}
\|H(0)\|_{X_\diamond} \leq \frac{1}{2(4C+1)} = \frac{R}{2}.
\end{equation*}

\n Now, for any $\bm{u} \in \overline{B_R}(0)$

\begin{equation*}
\|H(\bm{u})\|_{X_\diamond} \leq \|H(\bm{u}) - H(0)\|_{X_\diamond} + \|H(0)\|_{X_\diamond} \leq \frac{1}{2} \|\bm{u}\|_{X_\diamond} + \frac{R}{2} \leq R.
\end{equation*}

\n Thus, $H$ is a contraction on $\overline{B_R}(0) \subset X_\diamond$, and maps this closed ball into itself for each $\omega \in \CG_j$. Therefore, by the Banach fixed point theorem, the equation $\bm{u} = H(\bm{u})$ or equivalently $T^\diamond(\bm{p}_\omega + \bm{u}) = \bm{f}$ has a unique solution $\bm{u} \in \overline{B_R}(0)$ for each $\omega \in \CG_j$. Since this argument applies to all $j=1,...,s$, we obtain the existence for any $\omega \in \CG$.

\section*{Acknowledgments}
{The research of FC was partially supported  by the AFOSR Grant  FA9550-23-1-0256 and  NSF Grant DMS-21-06255.  ML was supported by the PDE-Inverse project of the European Research Council of the European Union and Academy of Finland grants 273979 and 284715. The research of MSV was partially supported by NSF Grant DMS-22-05912. Views and opinions expressed are those of the authors only and do not necessarily reflect those of the European Union or the other funding organizations.}

\appendix

\section{Appendix}
\setcounter{equation}{0} 

Here we present the proofs of Lemmas~\ref{LEM PN_omega}, \ref{LEM gamma_omega} and \ref{LEM v_omega - v_0}. We start with an auxiliary estimate for the solution of a second order differential equation in the presence of a small parameter. Recall that $Q$ is given by \eqref{q assumption} and define

\begin{equation*}
\Omega_Q = \left\{ \omega \in \CC: |\omega| \leq \frac{1}{2 \sqrt{Q}} \right\},
\end{equation*}

\begin{lemma} \label{LEM y}
Let $a,b \in \RR$ be fixed, $\omega \in \Omega_Q$ and $q$ satisfy \eqref{q assumption}. Let $z = z_\omega$ be the solution of the following initial value problem 

\begin{equation} \label{y}
\begin{cases}
z'' + \omega^2 q z = 0, \qquad \text{in} \ D = (0,1)
\\
z(0) = b, 
\\
z'(0) = a.
\end{cases}
\end{equation}

\n Then there exists a constant $C=C(a,b,Q)>0$, such that

\begin{equation} \label{y_omega ax+b expansion}
z_\omega(x) = ax + b + \omega^2 r_\omega(x),
\qquad \hbox{with }
\|r_\omega\|_{H^2} \leq C.
\end{equation}

\n Furthermore, given any  $\omega_0 \in \Omega_Q$

\begin{equation} \label{y_omega - y_omega0}
\|z_\omega - z_{\omega_0}\|_{H^2} \leq C |\omega - \omega_0|.
\end{equation}

\n In particular, $\omega \mapsto z_\omega$ is continuous as a mapping $\Omega_Q \to H^2$.

\end{lemma}

\begin{proof}

Clearly, $z_\omega$ satisfies the following integral equation
    
\begin{equation} \label{y inteq}
z_\omega(x) = ax+b + \omega^2 \int_0^x (t-x) q(t) z_\omega(t) dt.
\end{equation}

\n This implies the estimate

\begin{equation*}
\|z_\omega\|_{L^\infty} \leq |a| + |b| + \frac12 |\omega|^2 \|q\|_{L^\infty} \|z_\omega\|_{L^\infty} \leq |a| + |b| + \frac{1}{8} \|z_\omega\|_{L^\infty},
\end{equation*}

\n where in the last step we used the hypothesis on $\omega$. Consequently, $\|z_\omega\|_{L^\infty} \leq \frac{8}{7} \left( |a| + |b| \right)$. But \eqref{y inteq} shows that

\begin{equation} \label{y tilde integral}
r_\omega (x) = \int_0^x (t-x) q(t) z_\omega(t) dt.
\end{equation}

\n In particular $r_\omega'' = - q z_\omega$,and  hence the $H^2$-norm of $r_\omega$ can be controlled by the $L^\infty$-norm of $z_\omega$. This verifies \eqref{y_omega ax+b expansion}. To prove the estimate \eqref{y_omega - y_omega0}, we write

\begin{equation*}
z_\omega(x) - z_{\omega_0}(x) = \left( \omega^2 - \omega_0^2 \right) \int_0^x (t-x) q(t) z_\omega(t) dt + \omega_0^2 \int_0^x (t-x) q(t) \left( z_\omega(t) - z_{\omega_0}(t) \right) dt.
\end{equation*}

\n Analogously to before we use this identity to obtain an $L^\infty$ bound for $z_\omega -z_{\omega_0}$, and subsequently an $H^2$ bound, both of the order $|\omega -\omega_0|$.

\end{proof}

\subsection{Proof of Lemma~\ref{LEM PN_omega}}

Throughout the proof we assume that $\omega \in \Omega_Q$. If $\{\bm{u}^1_\omega, \bm{u}^2_\omega\}$ is an $X$-orthonormal basis for $N_\omega$, then the orthogonal projection operator is given by the following formula:

\begin{equation*}
P_{N_\omega} \bm{u} = \sum_{j=1}^2 \langle \bm{u}, \bm{u}^j_\omega \rangle_X \ \bm{u}^j_\omega, \qquad \qquad \forall \ \bm{u} \in X.
\end{equation*}

\n For any $\omega, \omega_0 \in \CC$, we can now write

\begin{equation*}
\left(P_{N_{\omega}} - P_{N_{\omega_0}} \right) \bm{u} = \sum_{j=1}^2 \langle \bm{u},  \bm{u}^j_{\omega} - \bm{u}^j_{\omega_0} \rangle_X \ \bm{u}^j_{\omega} + \langle \bm{u}, \bm{u}^j_{\omega_0} \rangle_X \ \left( \bm{u}^j_{\omega} - \bm{u}^j_{\omega_0} \right),
\end{equation*}

\n which readily implies the bound

\begin{equation*}
\left\| P_{N_{\omega}} - P_{N_{\omega_0}} \right\|_{\CL(X)} \leq 2
\sum_{j=1}^2 \left\| \bm{u}^j_{\omega} - \bm{u}^j_{\omega_0} \right\|_X. 
\end{equation*}

\n Therefore, to establish the first part of Lemma~\ref{LEM PN_omega} it suffices to construct a family of orthonormal bases $\{\bm{u}^1_\omega, \bm{u}^2_\omega\}$ for which the mappings $\omega \mapsto \bm{u}^j_\omega$, $j=1,2$ are continuous $\Omega_Q \to X$. Recall that $X = H^2 \times L^2$. Let $z_\omega, \widetilde{z}_\omega$ be given by \eqref{z_omega}, and define

\begin{equation*}
\bm{u}^1_\omega = (u_\omega, 0) ,
\qquad \qquad u_\omega = \frac{\widetilde{z}_\omega}{\|\widetilde{z}_\omega\|_{H^2}}
\end{equation*}

\n and

\begin{equation*}
\bm{u}^2_\omega(x) = (v_\omega, 0)
\qquad \qquad
v_\omega(x) = \frac{f_\omega(x)}{\|f_\omega\|_{H^2}},
\end{equation*}

\n where

\begin{equation*}
f_\omega(x) = z_\omega(x) - \langle z_\omega, u_\omega  \rangle_{H^2} \ u_\omega(x).
\end{equation*}

\n The continuity of $\omega \mapsto \bm{u}^j_\omega$ now follows from the continuity of $z_\omega, \ \widetilde{z}_\omega$, with respect to $\omega$, as shown in Lemma~\ref{LEM y} (and the fact that $\Vert \tilde z_\omega \Vert_{H^2}$ and $\Vert f_\omega \Vert_{H^2}$ stay bounded away from zero). This completes the proof of the first part of Lemma~\ref{LEM PN_omega}. 

Let us now turn to the proof of inequality \eqref{P_omega-P0 bound omega^2}, which is concerned with an estimate for small $\omega$. Lemma~\ref{LEM y} applied to $\widetilde{z}_\omega$ implies that

\begin{equation*}
\widetilde{z}_\omega = 1 + \omega^2 \widetilde{r}_\omega,
\qquad \hbox{with }
\|\widetilde{r}_\omega\|_{H^2} \leq C,
\end{equation*}

\n where the constant $C$ depends only on $Q$. This clearly shows that

\begin{equation*}
\|\widetilde{z}_\omega\|_{H^2} = 1 + O(\omega^2).
\end{equation*}

\n Combining the last two identities we also conclude that, there exist $\delta, C > 0$ (depending only on $Q$), such that for all $|\omega| < \delta$  

\begin{equation*}
u_\omega = 1 + \omega^2 \widetilde{u}_\omega,
\qquad \hbox{with }
\|\widetilde{u}_\omega\|_{H^2} \leq C,
\end{equation*}

\n and therefore

\begin{equation*}
\left\| \bm{u}^1_{\omega} - \bm{u}^1_{0} \right\|_X = \|u_\omega - 1\|_{H^2} \leq C |\omega|^2, \qquad \qquad \forall \ |\omega| < \delta.
\end{equation*}

\n Similarly, Lemma~\ref{LEM y} applied to $z_\omega$ implies that

\begin{equation*}
z_\omega = x + \omega^2 r_\omega,
\qquad \hbox{with }
\|r_\omega\|_{H^2} \leq C.
\end{equation*}

\n Using the above representation, the fact that $D = (0,1)$ and integrating out the $H^2$-inner product we find that, for $|\omega| < \delta$

\begin{equation*}
f_\omega(x) = x - \langle x, 1 \rangle_{H^2} + \omega^2 \widetilde{f}_\omega = x - \frac{1}{2} + \omega^2 \widetilde{f}_\omega,
\qquad \hbox{with }
\|\widetilde{f}_{\omega}\|_{H^2} \leq C,
\end{equation*}

\n where $C$ is a possibly different constant, which depends only on $Q$. Consequently, we have

\begin{equation*}
\|f_\omega\|_{H^2}  = \|x-\tfrac{1}{2}\|_{H^2} + O(\omega^2).
\end{equation*}

\n But then, as $\omega \to 0$

\begin{equation*}
\left\| \bm{u}^2_{\omega} - \bm{u}^2_{0} \right\|_X = \left\| v_\omega - \frac{x-\frac{1}{2}}{\|x-\tfrac{1}{2}\|_{H^2}} \right\|_{H^2}
= O(\omega^2).
\end{equation*}

\subsection{Proof of Lemma~\ref{LEM gamma_omega}}

\n By definition

\begin{equation*}
\begin{split}
\gamma_\omega &= \sup \left\{ \frac{\|L_\omega^{-1} \bm{v}\|_X}{\|\bm{v}\|_X}: \ \bm{v} \in R_\omega \backslash \{0\} \right\}
=
\sup \left\{ \frac{\|\bm{u}\|_X}{\|L_\omega \bm{u}\|_X}: \ \bm{u} \in N_\omega^\perp \backslash \{0\} \right\} =
\\[.1in]
&=\sup \left\{ \frac{1}{\|L_\omega \bm{u}\|_X}: \ \bm{u} \in N_\omega^\perp, \quad \|\bm{u}\|_X = 1 \right\},
\end{split}
\end{equation*}

\n hence

\begin{equation*}
\frac{1}{\gamma_\omega} = \inf \left\{ \|L_\omega \bm{u}\|_X: \ \bm{u} \in N_\omega^\perp, \quad \|\bm{u}\|_X = 1 \right\}. 
\end{equation*}

\n It is convenient to have a representation where $\bm{u}$ varies over $X$. To that end simply note that we also have

\begin{equation}
\label{gamma^-1}
\gamma_\omega^{-1} = \inf \left\{ \|L_\omega \bm{u}\|_X: \ \bm{u} \in X, \quad \|P_{N_\omega^\perp} \bm{u}\|_X=\| \bm{u} \|_X = 1 \right\} .
\end{equation}

\n Since $\gamma_\omega > 0$ for any $\omega \in \Omega_Q$ (recall that $L_\omega$ is a Fredholm operator according to Theorem~\ref{THM L}), we may equivalently show the continuity of $\gamma_\omega^{-1}$. Fix $\omega_0 \in \Omega_Q$ and let us prove the continuity at this point. Let $\omega \in \CC$ be close enough to $\omega_0$ so that

\begin{equation} \label{PN_omega^perp - P_N0^perp}
a_\omega:=\left\| P_{N_{\omega}^\perp} - P_{N_{\omega_0}^\perp} \right\|_{\CL(X)} \leq \frac{1}{2}.
\end{equation}

\n This is possible due to Lemma~\ref{LEM PN_omega} and the fact that $P_{N_{\omega}^\perp} = Id - P_{N_{\omega}}$. Take now any $\bm{u}$ that is admissible in \eqref{gamma^-1}, then

\begin{equation*}
1 = \|P_{N_\omega^\perp} \bm{u}\|_X \leq \|P_{N_{\omega_0}^\perp} \bm{u}\|_X + a_\omega.
\end{equation*} 

\n Therefore,

\begin{equation} \label{PN_0^perp u lower}
\|P_{N_{\omega_0}^\perp} \bm{u}\|_X \geq 1 - a_\omega.
\end{equation}

\n Because of \eqref{PN_omega^perp - P_N0^perp} the above norm is at least $1/2$, in particular $P_{N_{\omega_0}^\perp} \bm{u} \neq 0$. We set

\begin{equation*}
b_\omega:= \|L_\omega  - L_{\omega_0}  \|_{\CL(X)}. 
\end{equation*}

\n Now using the triangle inequality we can estimate

\begin{equation} \label{L_omega u lower bound}
\|L_\omega \bm{u}\|_X \geq \|L_{\omega_0} \bm{u}\|_X - b_\omega  = \|P_{N_{\omega_0}^\perp} \bm{u}\|_X \ \|L_{\omega_0} \bm{v}\|_X - b_\omega,     
\end{equation} 

\n where 

\begin{equation*}
\bm{v} = \frac{P_{N_{\omega_0}^\perp} \bm{u}}{\|P_{N_{\omega_0}^\perp} \bm{u}\|_X}.
\end{equation*}

\n Observe that $\|\bm{v}\|_X=\|P_{N_{\omega_0}^\perp} \bm{v}\|_X = 1$. In particular $\bm{v}$ is an admissible element in the infumum defining $\gamma_{\omega_0}^{-1}$, so that

\begin{equation*}
\|L_{\omega_0} \bm{v}\|_X \geq \gamma_{\omega_0}^{-1} .
\end{equation*}

\n Using this inequality, along with \eqref{PN_0^perp u lower} in \eqref{L_omega u lower bound} we obtain  

\begin{equation*}
\|L_\omega \bm{u}\|_X \geq \left( 1 - a_\omega \right) \gamma_{\omega_0}^{-1}  - b_\omega. 
\end{equation*}

\n Now we can take infimum over $\bm{u}$, and the left hand side will yield $\gamma_\omega^{-1}$. Hence

\begin{equation} \label{gammas 1}
\gamma_{\omega}^{-1} - \gamma_{\omega_0}^{-1} \geq - a_\omega \gamma_{\omega_0}^{-1}  - b_\omega. 
\end{equation}

\n Changing the roles of $\omega$ and $\omega_0$ and repeating the above argument we get

\begin{equation} \label{gammas 2}
\gamma_{\omega_0}^{-1} - \gamma_{\omega}^{-1} \geq - a_\omega \gamma_{\omega}^{-1}  - b_\omega. 
\end{equation}

\n As $a_\omega \leq 1/2$, the last inequality can be used to get an upper bound for $\gamma_\omega^{-1}$, namely

\begin{equation} \label{gamma_omega^-1 upper}
\frac{1}{2} \gamma_\omega^{-1} \leq (1-a_\omega) \gamma_\omega^{-1} \leq \gamma_{\omega_0}^{-1} + b_\omega.  
\end{equation}

\n Combining the estimates \eqref{gammas 1}, \eqref{gammas 2} and \eqref{gamma_omega^-1 upper} we obtain

\begin{equation*}
- a_\omega \gamma_{\omega_0}^{-1}  - b_\omega  \leq \gamma_{\omega}^{-1} - \gamma_{\omega_0}^{-1} \leq 2 a_\omega \left( \gamma_{\omega_0}^{-1} + b_\omega  \right) + b_\omega .
\end{equation*}

\n To complete the proof we use the fact that $a_\omega, b_\omega \to 0$ as $\omega \to \omega_0$. The convergence of $a_\omega$ follows from Lemma~\ref{LEM PN_omega}, whereas the convergence of $b_\omega$ follows directly from the definition (\ref{L}).

\subsection{Proof of Lemma~\ref{LEM v_omega - v_0}}

Consider the first inequality of Lemma~\ref{LEM v_omega - v_0}. To simplify notation we suppress the $\bm{\alpha}$ dependence and write $\bm{v}_\omega = \bm{v}_{\omega, \bm{\alpha}}$ (recall that $\bm{v}_\omega$ is defined in \eqref{v_omega basis expansion}). Since the second component of $\bm{v}_{\omega}$ is 0, we have

\begin{equation*}
\bm{v}_{\omega} =(v_\omega, 0), \qquad \qquad v_\omega(x) = \alpha_1 z_\omega(x) + \alpha_2 \widetilde{z}_\omega(x),
\end{equation*} 

\n where $z_\omega, \widetilde{z}_\omega$ are given by \eqref{z_omega}. Observe that $v_0$ is a linear function, hence

\begin{equation}
\label{secondder}
\left( v_{\omega} - v_0 \right)'' = v_{\omega}'' = - \omega^2 q v_{\omega} .
\end{equation}
Also, observe that $(v_\omega-v_0)(0)=(v_\omega-v_0)'(0)=0$, and therefore
$$
(v_\omega-v_0)(x)= \int_0^x (v_\omega-v_0)''(t)(x-t)dt~,
$$
from which it immediately follows that
$$
\int_0^1|v_\omega-v_0|^2 dx \le \frac1{12} \int_0^1|(v_\omega-v_0)''|^2 dx~. 
$$
In particular
$$
\Vert \bm{v}_\omega -\bm{v}_0\Vert_X^2 = \Vert v_\omega-v_0 \Vert^2_{L^2}+ \Vert (v_\omega-v_0)'' \Vert^2_{L^2} \le \frac{13}{12} \Vert (v_\omega-v_0)'' \Vert_{L^2}^2~.
$$
In combination with (\ref{secondder}) this now yields
$$
\Vert \bm{v}_\omega -\bm{v}_0\Vert_X \le \sqrt{\frac{13}{12}} Q |\omega|^2 \Vert v_\omega \Vert_{L^2}  \le \sqrt{\frac{13}{12}} Q |\omega|^2 \Vert \bm{v}_\omega \Vert_X~.
$$
The second inequality of Lemma~\ref{LEM v_omega - v_0} is an immediate consequence of the facts that
$\Vert \bm{v}_0 \Vert_X$ is bounded from below by  $c|\bm{\alpha}|$ and that $\Vert \bm{v}_\omega \Vert_X$ is bounded from above by $C|\bm{\alpha}|$ uniformly for $\omega \in \Omega_Q$. We leave the details to the reader.

\bibliographystyle{plain}
\bibliography{nonlinear}
\end{document}